\documentclass{amsart}
\usepackage{fullpage,graphicx,color,amsfonts,amssymb,amsmath,amsthm, bbm}
\usepackage[all]{xy}

\theoremstyle{plain} 
\newtheorem{theorem}    {Theorem}[section]
\newtheorem{theoremletter}{Theorem}

\newtheorem{lemma}      [theorem]{Lemma}
\newtheorem{corollary}  [theorem]{Corollary}

\newtheorem{proposition}[theorem]{Proposition}

\newtheorem{definition} [theorem]{Definition}
\newtheorem{remark}     [theorem]{Remark}

\newtheorem{example}    [theorem]{Example}

\theoremstyle{remark}

\newcommand{\End}{\operatorname{End}}
\newcommand{\Jac}{\operatorname{Jac}}
\newcommand{\spin}{\operatorname{spin}}
\newcommand{\Cl}{\operatorname{Cl}}
\newcommand{\sign}{\operatorname{sign}}
\newcommand{\Gal}{\operatorname{Gal}}
\newcommand{\Ind}{\operatorname{Ind}}

\newcommand{\Sym}{\operatorname{Sym}}
\newcommand{\Ad}{\operatorname{Ad}}
\newcommand{\GSp}{\operatorname{GSp}}

\newcommand{\GL}{\operatorname{GL}}
\renewcommand{\Re}{\operatorname{Re}}

\newcommand{\C}{\mathbb C}
\newcommand{\A}{\mathbb A}
\newcommand{\R}{\mathbb R}

\newcommand{\suml}{\sum\limits}

\title{Comparing Hecke Coefficients of Automorphic Representations}
\author{Liubomir Chiriac and Andrei Jorza}


\thanks{The second author is partially supported by NSA Grant H98230-16-1-0302.}

\address{University of Massachussetts Amherst, Department of Mathematics and Statistics, 710 N. Pleasant Street, Amherst, MA 01003}
\email{chiriac@math.umass.edu}

\address{University of Notre Dame, 275 Hurley Hall, Notre Dame, IN 46556}
\email{ajorza@nd.edu}

\begin{document}
\begin{abstract}
We prove a number of unconditional statistical results of the Hecke
coefficients for unitary cuspidal representations of $\GL(2)$
over number fields. Using partial bounds on the size of the Hecke coefficients, instances of Langlands functoriality, and properties of Rankin-Selberg $L$-functions, we obtain bounds on the set of places where linear combinations of Hecke coefficients are negative. Under a mild functoriality assumption we extend these methods to $\GL(n)$. As an application, we obtain a result related to a question of Serre about the occurrence of large Hecke eigenvalues of Maass forms. Furthermore, in the cases where the Ramanujan conjecture is satisfied, we obtain distributional results of the Hecke coefficients at places varying in certain congruence or Galois classes.
\end{abstract}

\maketitle

\section*{Introduction}
The statistical properties of Hecke eigenvalues of automorphic
forms represent a richly studied area at the intersection of analytic
and algebraic number theory. The Sato-Tate conjecture for
modular forms \cite{sato-tate-1,sato-tate:hmf}, a seminal result on
distributions of Hecke eigenvalues, was an important application of potential
modularity lifting \cite{sato-tate-2,sato-tate-3}. While much progress has been made recently to describe conjecturally the statistical distribution of the Hecke eigenvalues of a particular automorphic representation  \cite{banaszak-kedlaya}, little is known for general reductive groups and general number fields, and potential modularity liftings are not suitable for general Sato-Tate conjectures. Even in the case of $\GL(2)$ over totally real fields where the Sato-Tate conjecture is proved for Hilbert modular forms of regular weight, only the joint distribution of at most two Hilbert modular forms is known \cite{harris:sato-tate}. 

The present article answers a number of questions regarding the
distribution of Hecke coefficients of unitary cuspidal
representations under three levels of generality. For arbitrary
unitary cuspidal representations of $\GL(2)$ over arbitrary
number fields $F$ we obtain density bounds on the number of places
where linear combinations of Hecke coefficients are
bounded. This level of generality includes Maass forms, for
which no connection to Galois representations is yet known. In
the case of unitary cuspidal representations $\pi$ of $\GL(n, \mathbb{A}_F)$ we obtain similar distributional
results under a mild functoriality assumption, namely that
$\pi\otimes\pi^\vee$ is automorphic. As a consequence we obtain
distributions of large Hecke coefficients on $\GL(2)$. Finally, we obtain stronger density bounds in the context of unitary cuspidal
representations that satisfy the
Ramanujan conjecture.

A refined multiplicity one conjecture due to Ramakrishnan \cite{ramakrishnan:pure} asserts that if $\pi$ and $\pi'$ are
cuspidal autmorphic representations of $\GL(n,\A_F)$ with the
property that $\pi_{v}\simeq \pi'_{v}$, for all finite places
$v$ outside a set of Dirichlet density $<1/2n^2$ then $\pi\simeq
\pi'$. The statement would follow from the Ramanujan conjecture
on $\GL(n)$, but this is currently unavailable beyond $n=1$. The
only other case that is known unconditionally is $n=2$, which
was proved by Ramakrishnan \cite{ramakrishnan:taylor} using the
Gelbart-Jacquet lift $\Ad(\pi)$. On the Galois side, however,
Ramakrishnan's conjecture follows from Serre's unitary trick on
$\ell$-adic representations (cf. \cite{rajan:galois}).

As a consequence, when $\pi_1$ and $\pi_2$ are non-isomorphic cuspidal representations of $\GL(2,\A_F)$ with trivial character (thus, with real Hecke coefficients) the set $\{v\mid a_v(\pi_1) \neq a_v(\pi_2)\}$ has lower Dirichlet density at least $1/8$, and it is reasonable to expect that there is no bias towards one of the representations. The bound $1/8$ is sharp and can only be achieved by pairs of dihedral representations \cite{walji:multiplicity-one}. When the $\pi_i$'s correspond to distinct classical holomorphic cusp forms, the first author proved in \cite{chiriac:hecke} that the set $\{p\mid a_p(f)<a_p(g)\}$ has density at least $1/16$. If the forms do not have complex multiplication and have weight at least 2, the Sato-Tate conjecture implies that for exactly half of the primes the corresponding Hecke eigenvalues are negative. Similarly, a joint version of the Sato-Tate conjecture due to Harris \cite{harris:sato-tate}, gives that the set $\{p\mid a_p(f)<a_p(g)\}$ has density $1/2$.

A natural question that arises in this context is whether similar estimates occur for more general unitary cuspidal representations over number fields. Our first result (Corollary \ref{c:hecke-gl2}) gives a positive answer in this direction.

\begin{theoremletter}\label{TA}
Let $F$ be a number field and $\pi_1, \pi_2$ two non-isomorphic
unitary cuspidal representations on $\GL(2,\mathbb{A}_F)$ which
have trivial central character and which are not solvable
polyhedral. Then the set
\begin{enumerate}
\item $\{v\mid a_v(\pi_i)<0\}$ has upper Dirichlet density at least $0.1118$.
\item $\{v\mid a_v(\pi_1)<a_v(\pi_2)\}$ has upper Dirichlet density at least $0.0414$.
\end{enumerate}
\end{theoremletter}

Our approach to Theorem \ref{TA} relies on replacing the bound $|a_v(\pi)|\leq 2$ predicted  by the Ramanujan conjecture with a partial bound due to Ramakrishnan \cite[Theorem A]{ramakrishnan:coefficients}, which states that $|a_v(\pi)|\leq X$ for $v$ in a set of places of lower Dirichlet density at least $1-1/X^2$, for any value of $X\geq 1$. A careful analysis of the known instances of Langlands functoriality combined with several combinatorial arguments lead to an optimization problem (cf. Theorem~\ref{t:no-rc-gl2}) that can be solved via a number of computations in Sage \cite{sage}.

As a corollary to Theorem \ref{TA} we obtain a distributional result on the Hecke coefficients of Maass forms varying in certain congruence classes. For example, if $g$ is a Maass eigenform on $\GL(2,\mathbb{A}_\mathbb{Q})$ which is not of solvable polyhedral type then (cf. Remark \ref{rem-after-ThA}):
\[\dfrac{\overline{\delta}(\{p\mid a_p(g)<0, p\equiv
  1\pmod{8}\})}{\delta(\{p\equiv 1\pmod{8}\})}\geq 0.0625,\]
where by $\delta$ and $\overline{\delta}$ we denote the Dirichlet density and the upper Dirichlet density, respectively.

In a letter to Shahidi \cite[Appendix]{shahidi:symmetric} Serre asked a question about the occurrence of large/small Hecke eigenvalues of Maass forms. This question has been studied by Walji \cite{walji:large} who obtains a lower bound of $a_v(\pi)\geq 0.788\ldots$ on a set of places density at least $0.01$. Our second result (Corollary \ref{c:walji}) is concerned with the large values of $|a_v(\pi)|$.

\begin{theoremletter}\label{TB}
Let $\pi$ be a unitary cuspidal representation of $\GL(2, \mathbb{A}_F)$ which has trivial central character and is not solvable polyhedral. Then $\{v\mid |a_v(\pi)|>1\}$ has upper Dirichlet density at least $0.001355$.
\end{theoremletter}

To prove Theorem \ref{TB} we first generalize Theorem \ref{TA} to unitary cuspidal representations of $\GL(n, \mathbb{A}_F)$ satisfying a functoriality assumption (cf. Theorem \ref{t:no-rc-gln}) and obtain a lower density bound on the occurence of negative Hecke coefficients in terms of the number of poles of a Rankin-Selberg $L$-function. This assumption is satisfied by the cuspidal representation $\Sym^2\pi$ of $\GL(3,\mathbb{A}_F)$ and we remark that $|a_v(\pi)|>1$ is equivalent to $a_v(\Sym^2\pi)>0$.

Finally, when $\pi$ is a regular algebraic unitary cuspidal representation occurring in the cohomology of a Shimura variety then $\pi$ is known to satisfy the Ramanujan conjecture (cf. \S\ref{s:gln}) and we obtain improved bounds on the distribution of Hecke coefficients $a_v(\pi)$ as $v$ varies in certain Galois conjugacy classes (cf. Corollaries \ref{c:hecke-congruences} and \ref{c:hecke-congruences-quadratic}). 

\begin{theoremletter}\label{cl:hecke-cong}
Let $f$ be a non-CM regular weight Hilbert cuspidal eigenform with trivial character. For any coprime ideals $\mathfrak{a}$
and $\mathfrak{m}$:
\[\frac{\overline{\delta}(\{v\mid a_v(f) < 0, v\equiv \mathfrak{a}
\pmod{\mathfrak{m}}\})}{\delta(\{v\equiv \mathfrak{a}\pmod{\mathfrak{m}}\})}\geq \frac{1}{8},\]
Moreover, if $E/F$ is any Galois extension such that there exists a quadratic subextension $L/F$ with $\Gal(E/L)$ abelian of order $n$, then:
\[\frac{\overline{\delta}(\{v\mid a_v(f) < 0, v\textrm{ splits completely in
}E\})}{\delta(\{v\textrm{ splits completely in
}E\})}\geq \frac{1}{16}.\]
\end{theoremletter}

Note that the density ratios in Theorem \ref{cl:hecke-cong} can be thought of as conditional probabilities. We prove Theorem \ref{cl:hecke-cong} in the more general context of unitary
cuspidal representations of $\GL(2, \mathbb{A}_F)$ satisfying
the Ramanujan conjecture. We remark that the Ramanujan conjecture on $\GL(2)$ is satisfied for regular algebraic cuspidal representations on $\GL(2)$ over totally real fields and has been proven recently over CM fields (see \S \ref{s:gln}). Our strategy uses Dirichlet's approach to primes in arithmetic progressions in the context of arbitrary linear combinations of Hecke coefficients of unitary cuspidal representations satisfying the Ramanujan conjecture from Theorem \ref{t:linear-comparison}.

The paper is organized as follows: in \S \ref{s:gln} we recall a
number of results on cuspidal representations on
$\GL(n,\mathbb{A}_F)$ related to Langlands functoriality and the
Ramanujan conjecture. In \S\ref{s:no-rc-gl2} we consider the case of $\GL(2)$ over number fields, and we deduce Theorem \ref{TA} (Corollary \ref{c:hecke-gl2}) from the optimization problem that we set up in Theorem \ref{t:no-rc-gl2}. Section \S\ref{s:no-rc-gln} is devoted to extending these methods for $\GL(n)$. The main technical result here is Theorem \ref{t:no-rc-gln}, which is followed by three applications, including Theorem \ref{TB} (Corollary \ref{c:walji}). In \S\ref{s:rc} we present several results that assume the Ramanujan conjecture, among which we mention Corollaries \ref{c:hecke-congruences} and \ref{c:hecke-congruences-quadratic} that comprise Theorem \ref{cl:hecke-cong}.

\section*{Notations}
For a number field $F$ we denote by $\mathbb{A}_F$ the ring of
adeles. For a finite place $v$ we denote by $F_v$ the $v$-adic completion of $F$ and by $\varpi_v$ a choice of uniformizer. An
automorphic representation $\pi$ of $\GL(n,\mathbb{A}_F)$ is a
restricted tensor product $\pi\simeq\bigotimes' \pi_v$ over the
places $v$ of $F$ where $\pi_v$ is a smooth representation of
$\GL(n,F_v)$.

When $v$ is a finite place of $F$ such that $\pi_v$ is an unramified representation of $\GL(n,F_v)$ we denote by $a_v(\pi)$ the sum of
the Satake parameters $\alpha_{v,j}$ of $\pi_v$. Explicitly, if
$\pi_v$ is the normalized induction of the character
$\chi_{v,1}\otimes\cdots\otimes\chi_{v,n}$ where
$\chi_{v,j}:F_v^\times\to \mathbb{C}^\times$ are continuous
unramified, then $\alpha_{v,j}=\chi_{v,j}(\varpi_v)$ and
$a_v(\pi)=\displaystyle\sum_{j=1}^n \alpha_{v,j}$. For convenience, if $k\geq 1$ we denote $a_{v,k}(\pi)=\displaystyle \sum_{j=1}^n \alpha_{v,j}^k$.

For a Dirichlet series $Z(s)$ which converges absolutely when $\Re s>1$ we denote
\[\overline{\mathcal{D}}(Z(s))=\limsup_{s\to
  1^+}\frac{Z(s)}{-\log(s-1)}.\]
When $\mathcal{F}$ is a set of
finite places of a number field $F$ we define
$\overline{\delta}(\mathcal{F})$
 as
$\overline{\mathcal{D}}(\sum\limits_{v\in
  \mathcal{F}}q_v^{-s})$.

If $(s_n)_{n\geq 1}$ is a sequence of reals in $(1,\infty)$ with
$\displaystyle \lim_{n\to\infty} s_n=1$ we denote
\[\overline{\mathcal{D}}_{(s_n)}(Z(s))=\limsup_{n\to
  \infty}\frac{Z(s_n)}{-\log(s_n-1)},\]
and $\overline{\delta}_{(s_n)}(\mathcal{F})=\overline{\mathcal{D}}_{(s_n)}(\sum\limits_{v\in
  \mathcal{F}}q_v^{-s})$. We denote $\delta_{(s_n)}(\mathcal{F})=\overline{\delta}_{(s_n)}(\mathcal{F})$ if the limit exists, in which case we say that $\mathcal{F}$ has $(s_n)$-Dirichlet density.

\section{Automorphic representations on $\GL(n)$ over number
  fields}\label{s:gln}
The Sato-Tate and Lang-Trotter conjectures describe statistical
behaviors of the Hecke coefficients $a_v(\pi)$ for a unitary
cuspidal representation $\pi$ of $\GL(n, \mathbb{A}_F)$, as $v$
varies among almost all places of a number field $F$ away from
the infinite places and the places where $\pi_v$ is ramified. In
this section we describe certain properties satisfied by these
Hecke coefficients that will be necessary for our applications.

First, if $\pi$ is unitary the complex conjugate $\overline{\pi}$ is naturally isomorphic to the dual $\pi^\vee$. As a result, if $\pi\simeq\pi^\vee$ then $a_v(\pi) = a_v(\overline{\pi})$ and therefore $a_v(\pi)\in \mathbb{R}$. When $\pi$ is a unitary representation on $\GL(2,\mathbb{A}_F)$, self-duality is equivalent to $\pi$ having quadratic central character, which must be trivial if $\pi$ is not dihedral. Remark that if $\pi$ is an algebraic representation such that all $a_v(\pi)$ are real numbers, the global Langlands conjecture would imply that $\pi\cong\pi^\vee$. We will mostly concern ourselves with self-dual representations for simplicity of statements, but in Theorem \ref{t:linear-comparison} we explain how our methods yield slightly weaker information in the context of non-self-dual representations.

The Ramanujan conjecture posits that for a unitary
cuspidal representation $\pi$ of $\GL(n, \mathbb{A}_F)$, where $F$ is
a number field, the local representation $\pi_v$ is a
tempered representation for every place $v$ of $F$. Equivalently, $|\alpha_{v,i}(\pi)|=1$ for every place $v$ and $1\leq i\leq n$, which implies that $|a_v(\pi)|\leq n$ for all $v$. 

Special cases of the Ramanujan conjecture are known, as follows:
\begin{enumerate}
\item If $F$ is a totally real field and $\pi$ is a unitary
cuspidal representation of $\GL(n,\mathbb{A}_F)$ which is
regular, algebraic and essentially self-dual. (This includes
cohomological Hilbert modular forms.)
\item If $F$ is a CM field and $\pi$ is a unitary cuspidal representation of $\GL(n, \mathbb{A}_F)$ which is
regular, algebraic and conjugate self-dual.
\item If $F$ is a CM field and $\pi$ is a unitary cuspidal representation of $\GL(2, \mathbb{A}_F)$ which is regular and algebraic, but without any assumptions on duality.
\item If $\pi$ is a unitary cuspidal representation on $\GL(n, \mathbb{A}_F)$ satisfying the Ramanujan conjecture then the base-change of $\pi$ to a number field $E/F$ will also satisfy the Ramanujan conjecture. 
\end{enumerate}
The first two cases are a consequence of the realization, due to Harris-Taylor \cite{harris-taylor}, of such automorphic forms in the cohomology of certain Shimura varieties. The third case is a recent stunning result by Allen, Calegari, Caraiani, Gee, Helm, Le Hung, Newton, Scholze, Taylor, and Thorne \cite{10-author}. 

For general unitary cuspidal representations and general number fields weaker bounds are known due to \cite{luo-rudnick-sarnak} with an improvement for smaller rank groups due to \cite{blomer-brumley}. Our general results rely on the following weaker version of the Ramanujan conjecture, due to Ramakrishnan \cite[proof of Theorem A]{ramakrishnan:coefficients}: if $\pi$ is a cuspidal unitary representation of $\GL(n, \mathbb{A}_F)$ and $X\geq 1$, then $|a_v(\pi)|\leq X$ for $v$ in a set of places of Dirichlet density at least $1-\frac{1}{X^2}$.

Finally, our main results make use of a series of deep results on
Langlands functoriality. If $\pi$ is a cuspidal representation
of $\GL(2,\mathbb{A}_F)$ then:
\begin{enumerate}
\item $\Sym^2\pi$ is an automorphic
representation of $\GL(3,\mathbb{A}_F)$ by
\cite{gelbart-jacquet}. It is cuspidal if and only if $\pi$ is
not dihedral, i.e., $\pi$ is not the automorphic induction of a
character from a quadratic extension.
\item $\Sym^3\pi$ is an automorphic representation of $\GL(4,
\mathbb{A}_F)$ by \cite{kim-shahidi}. It is cuspidal if and only
if $\pi$ is neither dihedral nor tetrahedral, i.e., $\pi$ does
not become dihedral after a cubic extension.
\item $\Sym^4\pi$ is an automorphic representation of $\GL(5,
\mathbb{A}_F)$ by \cite{kim:sym4}. It is cuspidal if and only
if $\pi$ is not dihedral, tetrahedral, or octahedral (i.e.,
$\pi$ becomes tetrahedral after a quadratic extension).
\item If $\sigma$ is another cuspidal representation of
$\GL(2,\mathbb{A}_F)$ then $\pi\otimes\sigma$ is an automorphic
representation of $\GL(4,\mathbb{A}_F)$ by
\cite{ramakrishnan:tensor}. It is cuspidal unless $\pi$
and $\sigma$ are twist equivalent.
\item  If $\sigma$ is a cuspidal representation of
$\GL(3,\mathbb{A}_F)$ then $\pi\otimes\sigma$ is an automorphic
representation of $\GL(6,\mathbb{A}_F)$ by
\cite{kim-shahidi}. It is cuspidal unless either $\pi$ is not
dihedral and $\sigma$ is a twist of $\Ad(\pi)$, or $\pi$ is
dihedral and $\sigma$ is the induction of a character on a cubic
field $E$, with $\pi_E$ Eisensteinian \cite{ramakrishnan-wang}.
\item If $\tau$ is a cuspidal representation of $\GL(4, \mathbb{A}_F)$ then $\wedge^2\tau$ is an automorphic representation of $\GL(6, \mathbb{A}_G)$ by \cite{kim:sym4}.
\end{enumerate}
We say that $\pi$ is solvable polyhedral if $\pi$ is dihedral, tetrahedral, or octahedral.

\section{Hecke coefficients for $\GL(2)$ over number fields}\label{s:no-rc-gl2}

We begin by generalizing results of \cite{chiriac:hecke} to
collections of cuspidal automorphic representations on $\GL(2)$ over arbitrary number fields. The principle underlying
our approach, already present in \cite{sound-paper} is the
following. Suppose $\{f_v\}$ are real numbers indexed by (almost
all) finite places of the number field $F$ and
$$\mathcal{F}=\{v\mid f_v<0\}.$$ Suppose, further, that
$|f_v|\leq B$ for a bound $B$ that does not depend on
$v$. Consider the Dirichlet series
\[Z(s) = \sum_v\frac{f_v^2}{q_v^s}.\]
When $v\notin \mathcal{F}$, $f_v\geq 0$ and therefore $f_v^2\leq
Bf_v$, whereas if $v\in \mathcal{F}$ then $-f_v\leq B$. As
\[Z(s)\leq B^2\sum_{v\in
  \mathcal{F}}\frac{1}{q_v^s}+B\left(\sum_v\frac{f_v}{q_v^s}-\sum_{v\in
  \mathcal{F}}\frac{f_v}{q_v^s}\right)\]
we conclude that $\displaystyle \overline{\mathcal{D}}(Z(s))\leq
2B^2 \overline{\delta}(\mathcal{F})+B
\overline{\mathcal{D}}(\sum_v\frac{f_v}{q_v^s})$. For real
numbers $m,M$ such that $m\leq \overline{\mathcal{D}}(Z(s))$ and
$M\geq \overline{\mathcal{D}}(\sum_v \frac{f_v}{q_v^s})$ we
conclude that
\[\overline{\delta}(\mathcal{F})\geq \frac{m-MB}{2B^2}.\]
In the context of this section, the last hypothesis of the
above approach is not satisfied and the technical difficulties lie
in using Langlands functoriality to control the resulting errors terms.

We note that solvable polyhedral representations are known to satisfy the Ramanujan conjecture, and we will consider such representations in \S \ref{s:rc}. In this section we will consider unitary cuspidal representations of $\GL(2,\mathbb{A}_F)$ which are not solvable polyhedral.

\begin{lemma}\label{l:rs-poles}
Suppose $\pi_1,\pi_2,\pi_3,\pi_4$ are twist inequivalent unitary cuspidal
representations of $\GL(2,\mathbb{A}_F)$ which are not solvable
polyhedral. Then the following
Rankin-Selberg $L$-functions have the following poles at $s=1$:
\begin{center}
\begin{tabular}{lll}
$L$-function & pole of order & \\
\hline
$L(\pi_i^{\otimes 2}\times \pi_i^{\vee\otimes 2},s)$ & 2 \\
$L(\pi_i^{\otimes 2}\times \pi_i^\vee\otimes\pi_j^\vee, s)$ & 0 \\
$L(\pi_i^{\otimes 2}\times\pi_j^{\vee\otimes 2}, s)$ & $\leq 1$  \\
$L(\pi_i\otimes\pi_i^\vee\times\pi_j\otimes\pi_j^\vee, s)$ & 1  \\
$L(\pi_i^{\otimes 2}\times \pi_j^\vee\otimes\pi_k^\vee,s)$ & 0  \\
$L(\pi_i\otimes \pi_i^\vee\times \pi_j\otimes\pi_k^\vee,s)$ & 0  \\
$L(\pi_1\otimes\pi_2\times\pi_3^\vee\otimes\pi_4^\vee, s)$ & $\leq 1$ \\
\end{tabular}
\end{center}
If $\pi_1,\ldots,\pi_4$ are not assumed to be twist inequivalent then all of the above $L$-functions have poles at $s=1$ of order $\leq 2$.

\end{lemma}
\begin{proof}
The first four lines in the table follow from \cite[Lemma
5]{walji:multiplicity-one}. Note that $$L(\pi_i^{\otimes 2}\times
\pi_j^\vee\otimes\pi_k^\vee,s) =
L(\omega_{\pi_i}\pi_j^\vee\times\pi_k^\vee,s)L(\Sym^2\pi_i\times
(\pi_j^\vee\otimes \pi_k^\vee), s).$$ Since the representations
are twist inequivalent it follows that
$\pi_j^\vee\otimes\pi_k^\vee$ is cuspidal and the result follows
from properties of Rankin-Selberg $L$-functions.

The second to last line follows from the previous as
$$\pi_i\otimes\pi_i^\vee\otimes\pi_j\otimes\pi_k^\vee\cong
\pi_i^{\otimes 2}\otimes
(\omega_{\pi_i}\pi_j^\vee)^\vee\otimes\pi_k^\vee,$$
and the last line follows from the fact that $\pi_1\otimes\pi_2$ and $\pi_3\otimes\pi_4$ are cuspidal representations on $\GL(4,\mathbb{A}_F)$.

When $\pi_1,\ldots,\pi_4$ are not assumed to be twist inequivalent the upper bound 2 follows as in the first part, using the fact that $\pi_i$ are not solvable polyhedral.
\end{proof}

\begin{lemma}\label{l:walji}
Let $\pi_1,\ldots, \pi_r$ be unitary cuspidal representations of
$\GL(2,\mathbb{A}_F)$. Let $\lambda_1,\ldots, \lambda_r\in
\mathbb{C}$ and $\mathcal{R}$ a set of places of $F$. Then for any sequence $(s_n)\to 1^+$ we have
\[\overline{\mathcal{D}}_{(s_n)}\left(\sum_{v\in \mathcal{R}}\frac{|\sum
  \lambda_i a_v(\pi_i)|^2}{q_v^s}\right)\leq
\overline{\delta}_{(s_n)}(\mathcal{R})^{1/2}T^{1/2}\textrm{    and    }\overline{\mathcal{D}}_{(s_n)}\left(\sum_{v\in \mathcal{R}}\frac{|\sum
  \lambda_i a_v(\pi_i)|}{q_v^s}\right)\leq
\overline{\delta}_{(s_n)}(\mathcal{R})^{3/4}T^{1/4},\]
where
$T = 2\sum\limits_i |\lambda_i|^4+6\sum\limits_{i<j}|\lambda_i \lambda_j|^2+24\sum\limits_{i<j<k<l}|\lambda_i \lambda_j \lambda_k \lambda_l|$ if the representations $\pi_1,\ldots, \pi_r$ are pairwise twist inequivalent and $T=2(\sum\limits_i |\lambda_i|)^4$ otherwise.
\end{lemma}
\begin{proof}
Let $S$ be a finite set of places such that $\pi_{i,v}$ is unramified for all $i$ and $v\notin S$. By the Cauchy-Schwarz inequality:
\begin{align*}
\overline{\mathcal{D}}_{(s_n)}\left(\sum_{v\in \mathcal{R}}\frac{|\sum
  \lambda_i a_v(\pi_i)|^2}{q_v^s}\right)^2&\leq \overline{\delta}_{(s_n)}(\mathcal{R})
\overline{\mathcal{D}}_{(s_n)}\left(\sum_{v\notin S}\frac{|\sum \lambda_i
  a_v(\pi_i)|^4}{q_v^s}\right).
\end{align*}
Since
\begin{align*}
\sum_{v\notin S}\frac{|\sum \lambda_i
  a_v(\pi_i)|^4}{q_v^s}&=\sum_{v\notin
  S}\frac{1}{q_v^s}\left(\sum
|\lambda_i|^4|a_v(\pi_i)|^4+\sum_{i\neq j}\lambda_i^2
\overline{\lambda_j}^2
a_v(\pi_i^{\otimes 2}\otimes\pi_j^{\vee \otimes 2})+2\sum_{i\neq j} \lambda_i^2
\overline{\lambda_i}\overline{\lambda_j}
a_v(\pi_i^{\otimes 2}\otimes\pi_i^\vee\otimes\pi_j^\vee)\right.\\
&\ +\left.\sum_{i,j,k\textrm{ distinct}}\lambda_i^2
\overline{\lambda_j}\overline{\lambda_k}a_v(\pi_i^{\otimes 2}\otimes\pi_j^\vee\otimes\pi_k^\vee)+2\sum_{i\neq
  j}\lambda_i \lambda_j \overline{\lambda_j}^2
a_v(\pi_i\otimes\pi_j\otimes\pi_j^{\vee \otimes 2})\right.\\
&\ +\left.2\sum_{i\neq
  j}|\lambda_i|^2|\lambda_j|^2a_v(\pi_i\otimes\pi_i^\vee\otimes\pi_j\otimes\pi_j^{\vee})
+ 4\sum_{i,j,k\textrm{ distinct}}|\lambda_i|^2 \lambda_j
\overline{\lambda_k}a_v(\pi_i\otimes\pi_i^\vee\otimes\pi_j\otimes\pi_k^\vee)\right.\\
&\ + 4\sum_{i,j,k\textrm{ distinct}}\lambda_i \lambda_j
\overline{\lambda_k}^2a_v(\pi_i\otimes\pi_j\otimes\pi_k^{\vee\otimes 2})+\left.\sum_{i,j,k,l\textrm{ distinct}}\lambda_i \lambda_j \overline{\lambda_k}\overline{\lambda_l}a_v(\pi_i\otimes\pi_j\otimes\pi_k^\vee\otimes\pi_l^\vee)\right),
\end{align*}
using Lemma \ref{l:rs-poles} (see also Lemma \ref{l:convergence-k2}) we get that
\begin{align*}
\overline{\mathcal{D}}_{(s_n)}\left(\sum_{v\in \mathcal{R}}\frac{|\sum
  \lambda_i a_v(\pi_i)|^2}{q_v^s}\right)^2&\leq \overline{\delta}_{(s_n)}(\mathcal{R})\cdot \begin{cases}
2\sum\limits_{i} |\lambda_i|^4 +3\sum\limits_{i\neq
  j}|\lambda_i \lambda_j|^2 +\sum\limits_{i,j,k,l\textrm{
    distinct}}|\lambda_i \lambda_j \lambda_k
\lambda_l|&\textrm{twist inequivalent}\\
2(\sum\limits_i |\lambda_i|)^4&\textrm{otherwise}\end{cases}
\end{align*}
Applying Cauchy-Schwarz again yields the second inequality.
\end{proof}

We now turn our attention to our main theorem for $\GL(2)$. We are grateful to the anonymous referee for pointing out that Ramakrishnan's result \cite[Theorem
A]{ramakrishnan:coefficients} did not apply in our setting as written. We therefore adapt Ramakrishnan's work to control Hecke operators in many intervals.

\begin{lemma}\label{l:density subsequences}
Let $\mathcal{F}_1,\ldots, \mathcal{F}_m$ be a finite collection of sets of finite places of a number field $F$ and $(a_n)_{n\geq 1}\to 1^+$ be any sequence. Then there exists a subsequence $(s_n)_{n\geq 1}\to 1^+$ of $(a_n)$ such that each set $\mathcal{F}_i$ has $(s_n)$-Dirichlet density.
\end{lemma}
\begin{proof}
We will prove the result by induction on $m$, in the base case
$m=0$ choosing $(s_n)=(a_n)$. Suppose
$\mathcal{F}_1,\ldots, \mathcal{F}_m$ all have $(s_n)$-Dirichlet
density and $\mathcal{F}_{m+1}$ is an additional set of finite
places. We choose $(s'_k)$ a subsequence of $(s_n)$ such that
\[\lim_{k\to\infty}\frac{\sum_{v\in
    \mathcal{F}_{m+1}}q_v^{-s'_k}}{-\log(s-1)}=\limsup_{n\to\infty}\frac{\sum_{v\in
    \mathcal{F}_{m+1}}q_v^{-s_n}}{-\log(s-1)}.\]
In this case $\mathcal{F}_i$ has $(s'_k)$-Dirichlet density for each $i\leq m+1$.
\end{proof}
The following lemma can be thought of as a version of \cite[Theorem
A]{ramakrishnan:coefficients}.
\begin{lemma}\label{l:ramakrishnan}
Let $F$ be a number field. For $1\leq i\leq r$ let $\pi_i$ be a cuspidal representation of
$\GL(n_i,\mathbb{A}_F)$, and let
$1<X=X_0<X_1<\ldots<X_m<\infty$. Then there exists a sequence of real numbers
$(s_n)_{n\geq 1}\to 1^+$ and sets $\mathcal{R}_j$ for $0\leq j\leq m$ and $\mathcal{R}_\infty$ of finite
places of $F$
such that if $v\in \mathcal{R}_j$ then $|a_v(\pi_i)|\leq X_j$ for all $i$, each set $\mathcal{R}_j$ with $j\geq 1$ has
$(s_n)$-Dirichlet density, and
\[\sum_{j=1}^m X_{j-1}^2\delta_{(s_n)}(\mathcal{R}_j)+X_m^2\delta_{(s_n)}(\mathcal{R}_\infty)\leq r.\]
\end{lemma}
\begin{proof}
For each $i$ we define $\mathcal{R}_{i,j}$ as the set of finite places $v$ of $F$ such that $X_{j-1}\leq |a_v(\pi_i)|\leq X_j$, with the convention that $X_{-1}=0$. We also denote  $\mathcal{R}_{i,\infty}$ the set of finite places $v$ of $F$ such that $X_{m}\leq |a_v(\pi_i)|$. As in the proof of \cite[Theorem
A]{ramakrishnan:coefficients} we have
\begin{equation}\label{eq:dinakar-ineq}
\log L(\pi_i\times\pi_i^\vee,s)\geq \sum_{j=1}^m\sum_{v\in
  \mathcal{R}_{i,j}}|a_v(\pi_i)|^2q_v^{-s}+\sum_{v\in \mathcal{R}_{i,\infty}}|a_v(\pi_i)|^2q_v^{-s}\geq
\sum_{j=1}^mX_{j-1}^2\sum_{v\in \mathcal{R}_{i,j}}q_v^{-s}+X_m^2\sum_{v\in \mathcal{R}_{i,\infty}}q_v^{-s}.
\end{equation}
We define recursively $\mathcal{R}_\infty=\bigcup_{i=1}^r \mathcal{R}_{i,\infty}$ and $\mathcal{R}_j = \left(\bigcup_{i=1}^r
\mathcal{R}_{i,j}\right)\setminus \left(\mathcal{R}_\infty\cup\bigcup_{k=j+1}^m\mathcal{R}_{k}\right)$ for each $j\leq m$.

Lemma \ref{l:density subsequences} implies the existence of a
sequence $(s_n)\to 1^+$ such that all of the above mentioned
sets have $(s_n)$-Dirichlet density. From
\eqref{eq:dinakar-ineq} we deduce that for each $i$:
\[\sum_{j=1}^m
X_{j-1}^2\delta_{(s_n)}(\mathcal{R}_{i,j})+X_m^2\delta_{(s_n)}(\mathcal{R}_{i,\infty})\leq \lim_{s\to 1^+}\frac{\log
  L(\pi_i\times\pi_i^\vee,s)}{-\log(s-1)}=1.\]
The desired result now follows from the fact that
$\delta_{(s_n)}(\mathcal{R}_j)\leq \suml_{i=1}^r\delta_{(s_n)}(\mathcal{R}_{i,j})$ for each $j\leq m$ or $j=\infty$.
\end{proof}

\begin{theorem}\label{t:no-rc-gl2}
Let $F$ be a number field and $\pi_1,\ldots, \pi_r$ be pairwise
non-isomorphic unitary cuspidal self-dual representations of
$\GL(2,\mathbb{A}_F)$ which are not solvable polyhedral. Let $\lambda_1,\ldots, \lambda_r$ be complex numbers, not all zero, such that $\sum\limits_{i=1}^r \lambda_i a_v(\pi_i)\in \mathbb{R}$ for all $v$. Then the set
$$\mathcal{F}=\{v\mid \sum_{i=1}^r \lambda_i a_v(\pi_i)< 0\}$$ has upper
Dirichlet density at least
\small
\[\max_{1<X<X_1<\ldots<X_m}\min_{\genfrac{}{}{0pt}{}{y_1,\ldots, y_m,y\geq 0}{\sum y_i+y\leq r}}\frac{\displaystyle A - \frac{\sqrt{yT}}{X_m} - \frac{ \sqrt[4]{Ty^3} B(X)}{X_m^{3/2}}-2 \sum_{k=1}^m\frac{(B(X_k)^{2}-B(X)^{2})y_k}{X_{k-1}^{2}}-\sqrt[4]{T} \sum_{k=1}^m \frac{(B(X_k)-B(X))y_k^{3/4}}{X_{k-1}^{3/2}}}{2 B(X)^2},\]
\normalsize
where $A=\sum\limits_i |\lambda_i|^2$, $B(x) = \sqrt{\frac{3+\sqrt{13+4x}}{2}}\sum\limits_{i}|\lambda_i|$, and $T = 2\sum\limits_i |\lambda_i|^4+6\sum\limits_{i<j}|\lambda_i \lambda_j|^2+24\sum\limits_{i<j<k<l}|\lambda_i \lambda_j \lambda_k \lambda_l|$ if the representations $\pi_1,\ldots, \pi_r$ are pairwise twist inequivalent and $T=2(\sum\limits_i |\lambda_i|)^4$ otherwise. We remark that for any $m$ the above maximum is positive.
\end{theorem}

\begin{proof}
Consider a sequence of real numbers $1<X=X_0<X_1<\ldots<X_m$. As
before, let $S$ be a finite set of places such that $\pi_{i,v}$
is unramified for all $i$ and $v\notin S$. The representation
$\Sym^4\pi_i$ is cuspidal on $\GL(5,\mathbb{A}_F)$ as $\pi_i$ is
not solvable polyhedral and therefore Lemma \ref{l:ramakrishnan}
implies the existence of a sequence $(s_n)\to 1^+$ and sets
$\mathcal{R}_0,\ldots, \mathcal{R}_m$ of finite places such that
each set $\mathcal{R}_j$ has $(s_n)$-Dirichlet density,
$|a_v(\Sym^4\pi_i)|\leq X_j$ for each $i$ and each $v\in
\mathcal{R}_j$, and
\[\sum_{j=1}^m X_{j-1}^2\delta_{(s_n)}(\mathcal{R}_j)+X_m^2\delta_{(s_n)}(\mathcal{R}_\infty)\leq r.\]
We will repeatedly use the fact that if $|a_v(\Sym^4\pi_i)|\leq X$ then $$|a_v(\pi)|\leq c(X)=\sqrt{\frac{3+\sqrt{13+4X}}{2}},$$ where $c(X)$ is the largest root of the polynomial $t^4-3t^2-1-X=0$. Indeed, this follows from the fact that $$|a_v(\Sym^4\pi)|=|a_v(\pi)^4-3a_v(\pi)^2+1|\geq |a_v(\pi)|^4-3|a_v(\pi)|^2-1.$$

We see that
\[\sum |\lambda_i|^2 \leq \overline{\mathcal{D}}_{(s_n)}\left(\sum_{v\notin
  S}\frac{|\sum \lambda_i a_v(\pi_i)|^2}{q_v^s}\right)\leq
 \overline{\mathcal{D}}_{(s_n)}\left(\sum_{k=0}^m\sum_{v\in \mathcal{R}_{k}}\frac{|\sum \lambda_i a_v(\pi_i)|^2}{q_v^s}\right)+\overline{\mathcal{D}}_{(s_n)}\left(\sum_{v\in \mathcal{R}_\infty}\frac{|\sum \lambda_i a_v(\pi_i)|^2}{q_v^s}\right),\]
At the same time, Lemma \ref{l:walji} implies that 
\[\overline{\mathcal{D}}_{(s_n)}\left(\sum_{v\in
  \mathcal{R}_\infty}\frac{|\sum \lambda_i a_v(\pi_i)|^2}{q_v^s}\right)\leq \delta_{(s_n)}(\mathcal{R}_\infty)^{1/2}T^{1/2}.\]
Using the idea from the the beginning of this section we get
\begin{align*}
\sum_{v\in\mathcal{R}_{k}}\frac{|\sum \lambda_i a_v(\pi_i)|^2}{q_v^s} &\leq 2 B(X_k)^2\sum_{v\in \mathcal{F}\cap \mathcal{R}_{k}}\frac{1}{q_v^s}+ B(X_k)\sum_{v\in \mathcal{R}_{k}}\frac{\sum \lambda_i a_v(\pi_i)}{q_v^s}.
\end{align*}
Putting these inequalities together and, if necessary, replacing $(s_n)$ with a subsequence in order to make each set $\mathcal{F}\cap \mathcal{R}_k$ have $(s_n)$-Dirichlet density yields
\begin{align*}
\overline{\mathcal{D}}_{(s_n)}\left(\sum_{k=0}^m \sum_{v\in \mathcal{R}_{k}}\frac{|\sum \lambda_i a_v(\pi_i)|^2}{q_v^2}\right)&\leq \sum_{k=0}^m 2B(X_k)^2\delta_{(s_n)}(\mathcal{F}\cap \mathcal{R}_{k}) +\overline{\mathcal{D}}_{(s_n)}\left(\sum_{k=0}^mB(X_k)\sum_{v\in \mathcal{R}_{k}}\frac{\sum \lambda_i a_v(\pi_i)}{q_v^s}\right).
\end{align*}
Again, using the fact that $\sum \lambda_i a_v(\pi_i)\leq |\sum
\lambda_i a_v(\pi_i)|$, Lemma \ref{l:walji} implies:
\small
\begin{align*}
\overline{\mathcal{D}}_{(s_n)}\left(\sum_{k=0}^mB(X_k)\sum_{v\in
  \mathcal{R}_{k}}\frac{\sum \lambda_i a_v(\pi_i)}{q_v^s}\right)&\leq
\overline{\mathcal{D}}_{(s_n)}\left(B(X)\sum_{v\in \cup
  \mathcal{R}_{k}}\frac{\sum \lambda_i a_v(\pi_i)}{q_v^s}\right) +
\sum_{k=1}^m \overline{\mathcal{D}}_{(s_n)}\left((B(X_k)-B(X))\sum_{v\in
  \mathcal{R}_{k}}\frac{\sum \lambda_i a_v(\pi_i)}{q_v^s}\right)\\
\leq &
\overline{\mathcal{D}}_{(s_n)}\left(B(X)\sum_{v\in 
  \mathcal{R}_\infty}\frac{|\sum \lambda_i a_v(\pi_i)|}{q_v^s}\right) +
\sum_{k=1}^m \overline{\mathcal{D}}_{(s_n)}\left((B(X_k)-B(X))\sum_{v\in
  \mathcal{R}_{X_k}}\frac{|\sum \lambda_i a_v(\pi_i)|}{q_v^s}\right)\\
\leq &
T^{1/4}\left(B(X) \delta_{(s_n)}(\mathcal{R}_\infty)^{3/4}+\sum_{k=1}^m(B(X_k)-B(X))\delta_{(s_n)}(\mathcal{R}_{X_k})^{3/4}\right).
\end{align*}
\normalsize
It follows that
\begin{align*}
\sum |\lambda_i|^2&\leq 2B(X)^2\delta_{(s_n)}(\mathcal{F}\setminus
  \mathcal{R}_\infty)+\sum_{k=1}^m2(B(X_k)^2-B(X)^2)\delta_{(s_n)}(\mathcal{F}\cap \mathcal{R}_{k})+\\
&\ +T^{1/4}B(X)\delta_{(s_n)}(\mathcal{R}_\infty)^{3/4}
+T^{1/4}\sum_{k=1}^m(B(X_k)-B(X))\delta_{(s_n)}(\mathcal{R}_{k})^{3/4}+\delta_{(s_n)}(\mathcal{R}_\infty)^{1/2}T^{1/2}
\end{align*}
In the above formula note that $\delta_{(s_n)}(\mathcal{F}\cap
\mathcal{R}_{k})\leq \delta_{(s_n)}(\mathcal{R}_{k})$ and $\delta_{(s_n)}(\mathcal{F}\setminus
\mathcal{R}_{\infty})\leq \delta_{(s_n)}(\mathcal{F})$. Writing
$y_k = X_{k-1}^2 \delta_{(s_n)}(\mathcal{R}_k)$ and $y=X_m^2
\delta_{(s_n)}(\mathcal{R}_\infty)$ we deduce that
\begin{align*}
\sum |\lambda_i|^2&\leq 2B(X)^2\delta_{(s_n)}(\mathcal{F})+\sum_{k=1}^m\frac{2(B(X_k)^2-B(X)^2)y_k}{X_{k-1}^2}+\\
&\ +\frac{T^{1/4}B(X)y^{3/4}}{X_m^{3/2}}
+T^{1/4}\sum_{k=1}^m\frac{(B(X_k)-B(X))y_k^{3/4}}{X_{k-1}^{3/2}}+\frac{T^{1/2}y^{1/2}}{X_m},
\end{align*}
for nonnegative real numbers $y_1,\ldots, y_m,y$ such that $\sum y_i+y \leq r$. The desired lower bound on $\overline{\delta}(\mathcal{F})\geq \delta_{(s_n)}(\mathcal{F})\geq \delta_{(s_n)}(\mathcal{F}\setminus \mathcal{R}_\infty)$ immediately follows.
\end{proof}

We remark that even for Hilbert modular forms our results do not
follow from the Sato-Tate conjecture as currently
available. Harris proved the Sato-Tate conjecture for pairs
$(f,g)$ of Hilbert modular forms using the Rankin-Selberg
$L$-functions $L(\Sym^m f\times\Sym^n g,s)$. Since
Rankin-Selberg $L$-functions of the form $L(\pi_1\times\pi_2\times\cdots\times\pi_r,s)$ are not currently available for more than two general cuspidal representations (except, in a small number of cases, on lower rank groups), it is not known that normalized Hecke coefficients of more than 2 Hilbert modular forms are distributed independently. In particular, Theorem \ref{t:no-rc-gl2} for more than 2 representations is new even in the context of Hilbert modular forms.

\bigskip

The first main result of this paper, Theorem \ref{TA}, can now readily be deduced from the above. 

\begin{corollary}\label{c:hecke-gl2}
Let  $\pi$ and $\sigma$ be twist inequivalent unitary cuspidal
representations of $\GL(2,\mathbb{A}_F)$ with trivial central
characters, which are not solvable polyhedral. Then: 
\begin{enumerate}
\item $a_v(\pi)<0$ for $v$ in a set of upper Dirichlet density
at least $0.1118$, and
\item $a_v(\pi)<a_v(\sigma)$ for $v$ in a set of upper Dirichlet density at least $0.0414$.
\end{enumerate}
\end{corollary}
\begin{proof}
In the statement of Theorem \ref{t:no-rc-gl2} take $r=1$,
$A=1$, $T=2$ and $X=3< 5< 8<17< 27< 38< 49< 61$ for the first part (for this choice of $X_i$ the minimum of the expression is attained at $y=0.0016$, $(y_k)=(0.86, 0.078, 0.055, 0.00012, 0.000011, 2.39e-6, 7.99e-7, 3.03e-7, 1.38e-7)$) and $r=2$, $A=2$, $T=10$, $X=10< 23< 30< 36< 45< 54< 72< 81< 90$ for the second part. These values were searched for, and the explicit lower bound was computed using Sage.
\end{proof}

\begin{remark}
Assuming the Ramanujan conjecture Theorem, \ref{t:linear-comparison} would yield the bounds $1/8=0.125$ instead of $0.1118$ and $1/16=0.0625$ instead of $0.0414$. For classical holomorphic modular forms, the bound $1/16$ was obtained by the first author in \cite[Theorem 1.1]{chiriac:hecke}.
\end{remark}

The ability to take arbitrary linear combinations of Hecke coefficients in Theorem \ref{t:no-rc-gl2} allows us to generalize Dirichlet's approach to primes in arithmetic progressions to obtain distribution results on the sign of $a_v(\pi)$ as $v$ varies in certain ray classes.

\begin{corollary}\label{c:hecke-congruences-gl2}
Let $F$ be a number field and $\pi$ a self-dual unitary cuspidal
automorphic representation of $\GL(2,\mathbb{A}_F)$ which is not
solvable polyhedral. For any coprime ideals $\mathfrak{a}$ and
$\mathfrak{m}$, $a_v(\pi)<0$ (or $>0$) for $v\equiv \mathfrak{a}
\pmod{\mathfrak{m}}$ varying in a set of places of upper
Dirichlet density at least
\small
\[\max_{1<X<X_1<\ldots<X_m}\min_{\genfrac{}{}{0pt}{}{y_1,\ldots, y_m,y\geq 0}{\sum y_i+y\leq h_{\mathfrak{m}}}}\frac{\displaystyle \frac{1}{h_{\mathfrak{m}}} - \frac{\sqrt{2y}}{X_m} - \frac{ \sqrt[4]{2y^3} c(X)}{X_m^{3/2}}-2 \sum_{k=1}^m\frac{(c(X_k)^{2}-c(X)^{2})y_k}{X_{k-1}^{2}}-\sqrt[4]{2} \sum_{k=1}^m \frac{(c(X_k)-c(X))y_k^{3/4}}{X_{k-1}^{3/2}}}{2 c(X)^2},\]
\normalsize
where $h_{\mathfrak{m}}$ is the narrow ray class number of conductor $\mathfrak{m}$ and $c(x)=\sqrt{\frac{3+\sqrt{13+4x}}{2}}$.
\end{corollary}
\begin{proof}
Let $\Cl(\mathfrak{m})$ be the narrow ray class group and
$H_{\mathfrak{m}}$ the associated class field. Consider every
character $\chi$ of
$\Gal(H_{\mathfrak{m}}/F)\cong\Cl(\mathfrak{m})$ as a Hecke
character of $F$. We will apply Theorem
\ref{t:no-rc-gl2} to the linear combination
\[f_v=\sum_{\chi:\Cl(\mathfrak{m})\to
  \mathbb{C}^\times}\frac{\chi(\mathfrak{a}^{-1})}{h_{\mathfrak{m}}}
a_v(\chi\pi) = \begin{cases}a_v(\pi)&\textrm{if }v\equiv
\mathfrak{a}\pmod{\mathfrak{m}}\\
0&\textrm{otherwise}\end{cases}.\]
The linear combination is a real number for all $v$ as $\pi$ is
self-dual. Moreover, since $\pi$ is not solvable polyhedral the
representations $\chi\pi$ are pairwise non-isomorphic as $\chi$
varies.

Remark that $f_v<0$ if and only if $a_v(\pi)<0$ and $v$ is in the class of $\mathfrak{a} \pmod{\mathfrak{m}}$. The lower bound on the upper Dirichlet density follows from Theorem \ref{t:no-rc-gl2} with 
$A=\sum |\lambda_i|^2 = h_{\mathfrak{m}}^{-1}$ and $T=2\left(\sum|\lambda_i|\right)^4=2$.
\end{proof}

\begin{remark} \label{rem-after-ThA}
For the explicit bound for Maass forms mentioned in the introduction after Theorem~\ref{TA}, namely that if $g$ is a Maass eigenform then $a_p(g)<0$ for $p\equiv 1\pmod{8}$ in a set of places with upper Dirichlet density at least $0.0156$ we take $h=4 = \varphi(8)$ and $X=19<40<69<98<127<156<185<214<243$ in the main formula above.
\end{remark}

\section{Hecke coefficients for $\GL(n)$ over number fields}\label{s:no-rc-gln}

The purpose of this section is to extend the method developed in \S\ref{s:no-rc-gl2} to the setting of $\GL(n)$ over arbitrary number fields $F$. We do that for a class of unitary cuspidal representations $\pi$ of $\GL(n,\mathbb{A}_F)$ which we call \textit{good} (see the definition below). For such $\pi$ the $L$-function $L(\pi^{\otimes 2}\times\pi^{\vee\otimes 2},s)=L((\pi\otimes\pi^\vee)\times(\pi\otimes\pi^\vee),s)$ is a product of Rankin-Selberg $L$-functions which converges when $\Re s>1$ and whose pole at $s=1$ has order which can be computed in terms of the decomposition of $\pi\otimes\pi^\vee$.

\begin{definition}\label{d}
A unitary automorphic representation $\pi$ of
$\GL(n,\mathbb{A}_F)$ is said to be a \textit{cuspidal sum}
if there exist unitary cuspidal representations $\pi_j$ of
$\GL(n_j,\A_F)$, $1\leq j \leq r$, such that
$\sum\limits_{j=1}^rn_j=n$ and for all $v$ in
$F$: $$a_v(\pi)=\sum_{j=1}^r a_v(\pi_j).$$
We say that $\pi$ is \textit{good} if $\pi\otimes\pi^\vee$ is a cuspidal sum.
\end{definition}

We remark that if $\pi$ is an essentially self-dual unitary cuspidal representation of $\GL(n,\mathbb{A}_F)$ such that $\Sym^2\pi$ and $\wedge^2\pi$ are cuspidal sums on $\GL \left(\binom{n+1}{2},\mathbb{A}_F \right)$, resp. $\GL \left( \binom{n}{2}, \mathbb{A}_F \right)$, then $\pi$ is good automorphic. Indeed, if $\pi^\vee\cong \pi\otimes \eta$ for some character $\eta$, then $$\pi\otimes\pi^\vee\simeq \left( \Sym^2\pi \otimes \eta \right) \oplus \left( \wedge^2\pi \otimes \eta \right).$$

We will need the following results about good representations:

\begin{lemma}\label{l:convergence-k2}
Let $\pi$ and $\sigma$ be good unitary cuspidal automorphic representations of $\GL(n, \mathbb{A}_F)$. Let $S$ be a finite set containing all the ramified places of $\pi$ and $\sigma$. Then 
\[\sum_{k\geq 2}\sum_{v\notin
  S}\frac{|a_{v,k}(\pi)|^2}{kq_v^{ks}}\text{ and } \sum_{k\geq 2}\sum_{v\notin S}\frac{|a_{v,k}(\pi)a_{v,k}(\sigma)|}{kq_v^{ks}} \]
converge when $\Re s\geq 1$. 
\end{lemma}

\begin{proof}
By assumption, there exist unitary cuspidal representations $\pi_i$ of $\GL(n_i,\mathbb{A}_F)$, $1\leq i \leq r$, such that $\sum\limits_{i=1}^r n_i = n^2$ and $a_v(\pi\otimes\pi^\vee)=\sum\limits_{i=1}^r a_v(\pi_i)$. Then 
\begin{align*}
|a_{v,k}(\pi)|^2 &= 
\left | a_{v,k}(\pi\otimes\pi^\vee) \right |=
\left |\sum_{i=1}^r a_{v,k}(\pi_i) \right |
\leq \sum_{i=1}^r |a_{v,k}(\pi_i)|.
\end{align*}
Using the Luo-Rudnick-Sarnak bound \cite{luo-rudnick-sarnak}, it follows that
$$|a_{v,k}(\pi_i)|\leq  n_i q_v^{k(1/2-1/(n_i^2+1))},$$
and therefore $$|a_{v,k}(\pi)|^2\leq n^2 q_v^{k(1/2-1/(n^4+1))}.$$
In conclusion
\[\sum_{k\geq 2}\sum_{v\notin
  S}\frac{|a_{v,k}(\pi)|^2}{kq_v^{ks}}\leq \sum_{k\geq
  2}\sum_{v\notin
  S}\frac{n^2}{kq_v^{k(s-1/2+2/(n^4+1))}}\]
which converges when $\Re s\geq 1$. The same argument applies for the 
second convergence.
\end{proof}

\begin{lemma}\label{l:cs}
Let $\pi$ be a good unitary essentially self-dual cuspidal representation of
$\GL(n,\mathbb{A}_F)$. Let $M$ be the order of the pole of the Rankin-Selberg $L$-function
$L(\pi^{\otimes 2}\times\pi^{\vee\otimes 2},s)$ at $s=1$. If $(s_n)\to 1^+$ is any sequence and
$\mathcal{R}$ is a set of places of $F$ with $\delta_{(s_n)}(\mathcal{R})=d$ then
\[\overline{\mathcal{D}}_{(s_n)}\left(\suml_{v\in
  \mathcal{R}}\frac{|a_v(\pi)|^2}{q_v^s}\right) \leq \sqrt{M d}\ \ \ \ \textrm{and}\ \ \ \ \overline{\mathcal{D}}_{(s_n)}\left(\sum_{v\in
  \mathcal{R}}\frac{|a_v(\pi)|}{q_v^s}\right) \leq \sqrt[4]{M d^3}.\]
\end{lemma}

\begin{proof}
Let $S$ be the union of the archimedean places of $F$ with the finite set of places where $\pi$ is ramified. Then
\[\sum_{v\notin S}\frac{|a_v(\pi)|^4}{q_v^s}\leq \sum_{v\notin S}\sum_{k\geq
  1}\frac{|a_{v,k}(\pi)|^4}{kq_v^{ks}}=\log
L_S(\pi^{\otimes 2}\times\pi^{\vee\otimes 2},s)\]
and therefore $\overline{\mathcal{D}}_{(s_n)}\left(\suml_{v\notin S}
\frac{|a_v(\pi)|^4}{q_v^s}\right)\leq M$. The statement
then follows from the Cauchy-Schwarz inequality:
\[\left(\sum\limits_{v\in \mathcal{R}}\frac{|a_v(\pi)|^2}{q_v^s}\right)^2\leq \left(\sum_v\frac{|a_v(\pi)|^4}{q_v^s}\right) \left(\sum_{v\in \mathcal{R}}\frac{1}{q_v^s}\right) \textrm{  and  }\left(\sum\limits_{v\in \mathcal{R}}\frac{|a_v(\pi)|}{q_v^s}\right)^2\leq \left(\sum_v\frac{|a_v(\pi)|^2}{q_v^s}\right)\left(\sum_{v\in \mathcal{R}}\frac{1}{q_v^s}\right).\]
\end{proof}

Now we are ready to state the main technical result of this section.

\begin{theorem}\label{t:no-rc-gln}
Let $\pi_1,\ldots, \pi_r$ be pairwise non-isomorphic good unitary cuspidal automorphic representation of $\GL(n_i,\mathbb{A}_F)$. For any $t\geq 0$, and any  $\lambda_1,\ldots, \lambda_r \in \C$ such that $\sum \limits_{i=1}^r \lambda_i a_v(\pi_i)\in \R$ for all $v$, the set $$\mathcal{F}=\left\{v\mid \sum_{i=1}^r \lambda_i
a_v(\pi_i) < -t\right\}$$ has upper Dirichlet density at least
\[\overline{\delta}(\mathcal{F})\geq \max_{X>1}\min_{0\leq y\leq rX^{-2}}\frac{t^2+A-(t+X B)(t(1-y)+y^{3/4}C)-(t^2+A)(y+y^{1/2}D )}{2(X B+t)^2},\]
where $A=\sum\limits_{i=1}^r |\lambda_i|^2$, $B=\sum\limits_{i=1}^r |\lambda_i|$, $C=\sum\limits_{i=1}^r |\lambda_i|\sqrt{M_i}$, $D=\sum\limits_{i=1}^r \sqrt{M_i}$, and $M_i$ is the order of the pole at $s=1$ of the Rankin-Selberg $L$-function
$L(\pi_i^{\otimes 2}\times\pi_i^{\vee\otimes 2},s)$. The bound is positive when $t$ is close to $0$, and when $t=0$ the minimum is achieved when $y=r/X^2$.
\end{theorem}

\begin{proof}
Set $\displaystyle Z_S(s) = \sum_{v\notin S} \frac{|\suml_{i=1}^r \lambda_i
a_v(\pi_i)+t|^2}{q_v^s}$, where $S$ is a finite set containing all the ramified places of the $\pi_i$'s and the archimedean places of $F$. By Lemma \ref{l:convergence-k2} we have
$$\overline{\mathcal{D}}(Z_S(s))=t^2+\overline{\mathcal{D}}(\sum \lambda_i \lambda_j\log L_S(\pi_i\times\pi_j^\vee,s))= t^2+\sum_{i=1}^r |\lambda_i|^2.$$
  
For $X>1$ we denote by $\mathcal{R}$ the set of finite places
$v\notin S$ such that $|a_v(\pi_i)|\leq X$ for all $i$, and we
denote by $\overline{\mathcal{R}}$ its complement, excluding the places in $S$. By
Ramakrishnan's estimate \cite[Theorem
A]{ramakrishnan:coefficients} we have $\displaystyle
\overline{\delta}(\overline{\mathcal{R}})\leq \frac{r}{X^2}$. As in the proof of Theorem \ref{t:no-rc-gl2} we choose a sequence $(s_n)\to 1^+$ such that $\mathcal{R}$ and $\overline{\mathcal{R}}$ have $(s_n)$-Dirichlet densities, in which case $\delta_{(s_n)}(\mathcal{R})+\delta_{(s_n)}(\overline{\mathcal{R}})=1$ and we denote $y=\delta_{(s_n)}(\overline{\mathcal{R}})\leq \dfrac{r}{X^2}$.

Using Lemma \ref{l:cs} we find that 
\begin{align*}
\overline{\mathcal{D}}_{(s_n)}\left(\suml_{v\in
  \overline{\mathcal{R}}}\frac{|\suml_{i=1}^r \lambda_i
  a_v(\pi_i)+t|^2}{q_v^s}\right)&\leq(t^2+\sum_{i=1}^r |\lambda_i|^2)\left( \overline{\mathcal{D}}_{(s_n)}\left(\sum_{v\in
  \overline{\mathcal{R}}}\frac{\suml_{i=1}^r |a_v(\pi_i)|^2}{q_v^s}\right)+\overline{\mathcal{D}}_{(s_n)}\left(\sum_{v\in
  \overline{\mathcal{R}}}\frac{1}{q_v^s}\right)\right)\\
&\leq (t^2+\sum_{i=1}^r |\lambda_i|^2)\left( \sum_{i=1}^r
\sqrt{M_i\delta_{(s_n)}(\overline{\mathcal{R}})}+\delta_{(s_n)}(\overline{\mathcal{R}})
\right)\\
&=(t^2+\sum_{i=1}^r |\lambda_i|^2)\left(\sqrt{y}\sum_{i=1}^r
\sqrt{M_i}+y \right).
\end{align*}
Note that if $v\in \overline{\mathcal{F}}\cap
\mathcal{R}$ then $|\suml_{i=1}^r \lambda_i a_v(\pi_i)+t|^2 \leq (X\suml_{i=1}^r
|\lambda_i|+t)(\suml_{i=1}^r \lambda_i a_v(\pi_i)+t),$ therefore
\begin{align*}
\sum_{v\in \overline{\mathcal{F}}\cap \mathcal{R}}\frac{|\suml_{i=1}^r \lambda_i a_v(\pi_i)+t|^2}{q_v^s}&\leq (X\sum_{i=1}^r
|\lambda_i|+t)\left(\sum_{v\in {\mathcal{R}}}\frac{\suml_{i=1}^r \lambda_i
  a_v(\pi_i)+t}{q_v^s} - \sum_{v\in \mathcal{F}\cap
  \mathcal{R}}\frac{\suml_{i=1}^r \lambda_i a_v(\pi_i)+t}{q_v^s}\right)
\end{align*}
Since $\suml_{v\notin S}\frac{a_v(\pi_i)}{q_v^s}$ converges at $s=1$ we have
\begin{align*}
\overline{\mathcal{D}}_{(s_n)}\left(\sum_{v\in {\mathcal{R}}}\frac{\suml_{i=1}^r \lambda_i
  a_v(\pi_i)+t}{q_v^s}\right)&=t \delta_{(s_n)}({\mathcal{R}}) + \overline{\mathcal{D}}_{(s_n)}\left(\sum_{v\in \overline{\mathcal{R}}}\frac{-\suml_{i=1}^r \lambda_i
  a_v(\pi_i)}{q_v^s}\right)\\
&\leq t\left(1-y\right) + y^{3/4}\sum |\lambda_i|\sqrt[4]{M_i},
\end{align*}
where the last inequality follows from Lemma \ref{l:cs}. The
desired inequality follows by putting everything together as for
all $X>1$ we have
\[t^2+A\leq (t+X B)(t(1-y)+y^{3/4}C)+2(X B+t)^2\delta_{(s_n)}(\mathcal{F}\cap \mathcal{R})+(t^2+A)(y+y^{1/2}D).\]
\end{proof}

Next, we exhibit a number of situations when the hypotheses of Theorem \ref{t:no-rc-gln} are satisfied. We begin with an application of Theorem \ref{t:no-rc-gln} to
the occurence of large Hecke coefficients of unitary cuspidal representations of $\GL(2)$ with trivial central characters, previously considered by Walji in  \cite{walji:large}.

\begin{corollary}\label{c:walji}
Let $\pi$ be a unitary cuspidal representation of $\GL(2,
\mathbb{A}_F)$ which has trivial central character and is not
solvable polyhedral. Then $|a_v(\pi)|>1$ for $v$ in a set of
places of upper Dirichlet density at least $0.001355$.

More generally, if $\lambda>0$ then
$|a_v(\pi)-\lambda|>\sqrt{1+\lambda^2}$ for $v$ varying in a set
of places of upper Dirichlet density at least
\[\max_{X>1}\frac{(1+4 \lambda^2)\left(1-(2+\sqrt{6})X^{-1}-2X^{-2}\right)-2^{3/4}(2 \lambda\sqrt{2}+\sqrt{3})(1+2 \lambda)X^{-1/2}}{2(1+2 \lambda)^2X^2}.\]
\end{corollary}
\begin{proof}

First, we remark that if $\pi$ is a good self-dual cuspidal
automorphic representation of $\GL(n, \mathbb{A}_F)$ the previous theorem implies that $a_v(\pi)<0$ (or $a_v(\pi)>0$)
for $v$ in a set of places of density at least
\[\max_{X>1}\left(\frac{1}{2X^2}-\frac{\sqrt{M}}{2X^{5/2}}-\frac{\sqrt{M}}{2X^3}-\frac{1}{2X^4}\right)>0,\]
where $M$ is the order of the pole of the Rankin-Selberg $L$-function $L(\pi^{\otimes 2}\times\pi^{\vee\otimes 2}, s)$. 

Let $\Pi$ be the symmetric square lift of 
$\pi$, which is an automorphic representation of
$\GL(3,\mathbb{A}_F)$. Since $\pi$ is not solvable polyhedral, both $\Sym^2 \pi$ and $\Sym^4 \pi$ are cuspidal. This means that the representation $$\Pi\otimes\Pi^\vee\simeq\Sym^4\pi\boxplus \Sym^2\pi\boxplus \mathbbm{1}$$ is good and therefore $\Pi$ satisfies the hypotheses of Theorem \ref{t:no-rc-gln} (here $\mathbbm{1}$ is the trivial automorphic representation of $\GL(1, \mathbb{A_F})$). The $L$-function $$L(\Pi^{\otimes 2}\times\Pi^{\vee\otimes 2}, s)= L\left((\Sym^4\pi\boxplus \Pi\boxplus
\mathbbm{1})\times(\Sym^4\pi\boxplus \Pi\boxplus \mathbbm{1}),s\right)$$ has a pole of order three at $s=1$. Therefore $a_v(\Pi)>0$ for $v$ in a set
of places with upper Dirichlet density at least
\[\max_{X>1}\left(\frac{1}{2X^2}-\frac{\sqrt{3}}{2X^{5/2}}-\frac{\sqrt{3}}{2X^3}-\frac{1}{2X^4}\right)>\frac{1.355}{1000}.\]
In this case $a_v(\Pi)=a_v(\pi)^2-1>0$ and therefore
$|a_v(\pi)|>1$. The value $0.001355$ can be obtained by
substituting $X=9.47$.

For the second part of the corollary note that $|a_v(\pi)-\lambda|>\sqrt{1+\lambda^2}$ if and only if $a_v(\Sym^2\pi)-2 \lambda a_v(\pi)>0$. The representations $\pi$ and $\Sym^2\pi$ are cuspidal and not isomorphic and the corresponding Rankin-Selberg poles are of order $2$ and $3$. Then lower density bound then follows from Theorem \ref{t:no-rc-gln}.
\end{proof}

In our second example we consider the automorphic induction from $\GL(2)$ to $\GL(4)$. 

\begin{corollary}\label{c:conjugates}
Let $E/F$ be a quadratic extension and $\pi$ a non-dihedral unitary cuspidal
representation of $\GL(2, \mathbb{A}_E)$ with trivial central
character and which is not the base change from $F$ of a representation of $\GL(2,\mathbb{A}_F)$. Then the set of {\bf split} places
$\mathcal{F}=\{v\textrm{ place of }F\textrm{ split in }E\mid v=w\cdot w^c, a_w(\pi)+a_{w^c}(\pi)<0\}$
has upper Dirichlet density at least
\[\overline{\delta}(\mathcal{F})\geq \max_{X>1}\left(\frac{1}{2X^2}-\frac{\sqrt{7}}{2X^{5/2}}-\frac{1}{2X^4}-\frac{\sqrt{7}}{2X^3}\right)>3.49\cdot 10^{-4}.\]
\end{corollary}

\begin{proof}
Consider the unitary cuspidal representation
$\Pi=\Ind_{E/F}\pi$. In this case
\[\Pi\otimes \Pi^\vee\simeq
\mathbbm{1}\boxplus\theta\boxplus\Ind_{E/F}\Sym^2\pi\boxplus\pi\otimes\pi^c\boxplus
\theta\pi\otimes\pi^c,\]
where
$\theta$ is the quadratic character defining $E/F$ and $\pi^c$
is the conjugate representation. The representation $\pi\otimes\pi^c$ a priori defined over $E$ descends to an automorphic representation over $F$ and therefore in the expression above we treat $\mathbbm{1},\theta, \Ind_{E/F}\Sym^2\pi$, $\pi\otimes\pi^c$ and $\theta\pi\otimes\pi^c$ as representations over $F$.

We will apply Theorem
\ref{t:no-rc-gln} in the case of $\Pi$ to study $a_v(\Pi)$,
which equals $0$ when $v$ is not split in $E$, and equals
$a_w(\pi)+a_{w^c}(\pi)$ when $v=w\cdot w^c$ is split in $E$.

The assumption on $\pi$ implies that $\pi^\vee\not\simeq\pi^c$ and therefore $\pi\otimes\pi^c$ is cuspidal on $\GL(4,\mathbb{A}_E)$. We deduce that the order of the pole at $s=1$ of the Rankin-Selberg $L$-function $L(\Pi^{\otimes 2}\times \Pi^{\vee\otimes 2},s)$ is at most $1^2+1^2+1^2+2^2=7$ as among $\mathbbm{1},\theta,\Ind\Sym^2\pi,\pi\otimes\pi^c, \theta\pi\otimes\pi^c$ the first four are not isomorphic. The lower bound then follows from Theorem \ref{t:no-rc-gln}. The explicit lower bound can be obtained by taking $X=18$.
\end{proof}

Our final application of Theorem \ref{t:no-rc-gln} is a generalization of Corollary \ref{c:hecke-congruences-gl2} to places $v$ which split completely in certain Galois extensions.
\begin{corollary}
Let $\pi$ be a non-dihedral unitary cuspidal representation of
$\GL(2,\mathbb{A}_F)$ with trivial central character. Suppose
$E/F$ is a finite Galois extension and $L/F$ is a quadratic
subextension such that $\Gal(E/L)$ is an abelian group of order
$n$.
Then $a_v(\pi)<0$ (or $>0$) for $v$ in a set of places of $F$ that {\bf split completely} in
$E$ of upper Dirichlet density at least
\[
\max_{X>1}\frac{(n+1)(1-\frac{n+1}{X^2})-(n+1)^{7/4}(2\sqrt{2}+(n-1)\sqrt{19})X^{-1/2}-\sqrt{n+1}(2\sqrt{2}+(n-1)\sqrt{19})X^{-1}}{2(n+1)^2X^2}>0
.\]
\end{corollary}

\begin{proof}
For each character
$\chi$ of $\Gal(E/L)$ consider the automorphic representation
$\sigma_\chi$ of $\GL(2,\mathbb{A}_F)$ given by
$\sigma_\chi=\Ind_{L/F}\chi$. When $\chi\neq 1$ the
representation $\sigma_\chi$ is cuspidal and we denote
$\Pi_\chi=\pi\otimes \sigma_\chi$ the cuspidal representation of
$\GL(4,\mathbb{A}_F)$.

If $\theta$ is the quadratic character associated to $L/F$ then:
\[a_v(\pi)+a_v(\theta\pi)+\sum_{\chi\neq
  1}a_v(\Pi_\chi)=\begin{cases}2na_v(\pi)&v\textrm{ splits
  completely in }E\\0&\textrm{otherwise}\end{cases},\]
is a real number for all $v$.

We will apply Theorem \ref{t:no-rc-gln}
to $\pi$, $\theta\pi$ (where $\theta$ is the quadratic character of $L/F$), and to the representations
$\Pi_\chi$, as $\chi$ varies among the nontrivial characters of
$\Gal(E/L)$, grouped in isomorphism classes, all representations
appearing with coefficient $1$. Note that
\[\Pi_{\chi}\otimes\Pi_{\chi}^\vee \simeq \mathbbm{1}\boxplus \theta\boxplus
\Sym^2\pi\boxplus \theta\Sym^2\pi \boxplus \Ind\chi^2\boxplus
\Ind\chi^2\otimes\Sym^2\pi.\]
Therefore $\Pi\otimes\Pi^\vee$ is good since the representation $\Ind\chi^2\otimes\Sym^2\pi$ is either cuspidal or an isobaric sum of type $(3,3)$ by \cite[Theorem A]{ramakrishnan-wang}. Moreover, the order of the pole at $s=1$ of the Rankin-Selberg $L$-function $L(\Pi_\chi^{\otimes 2}\times\Pi_\chi^{\vee\otimes 2},s)$ is at most $1^2+1^2+1^2+4^2=19$, while the order of the pole at $s=1$ of $L((\chi\pi)^{\otimes 2}\times (\chi\pi)^{\vee \otimes 2},s)$ is 2. Hence, we see that 
$A\geq n+1$, $B=n+1$,
$C\leq 2\sqrt{2}+(n-1)\sqrt{19}$, and
$D\leq 2\sqrt{2}+(n-1)\sqrt{19}$.

Theorem \ref{t:no-rc-gln} implies that for $v$ in a set of
places of upper Dirichlet density at least the RHS in the 
desired lower bound we have
\[a_v(\pi)+a_v(\theta\pi)+\sum_{\chi\neq 1}a_v(\pi\otimes \sigma_{\chi})<0.\]
However, this can only happen when $v$ splits completely in $E$ in which case we see that $a_v(\pi)<0$ as well.

\end{proof}

\section{Comparing Hecke coefficients when the Ramanujan conjecture holds}\label{s:rc}
In this section we give improved density bounds in the context of unitary cuspidal representations that satisfy the Ramanujan conjecture. This includes, for example, regular algebraic cuspidal representations that show up in the cohomology of Shimura varieties.

\begin{theorem}\label{t:linear-comparison}
Let $F$ be a number field. For each $1\leq i\leq r$ let $\pi_i$
be pairwise non-isomorphic unitary cuspidal automorphic representation of $\GL(n_i,\mathbb{A}_F)$ that satisfy the Ramanujan conjecture. Fix $\lambda_1,\ldots, \lambda_r\in \mathbb{C}$, not all zero, $t\in \mathbb{C}$, and $\epsilon\in (0,\pi/2)$. Then:
\begin{enumerate}
\item The set $\mathcal{F}=\{v\mid \arg(\sum \lambda_i a_v(\pi_i)-t)\notin
(-\epsilon,\epsilon)\}$ has upper Dirichlet density
\[\overline{\delta}(\mathcal{F})\geq \dfrac{|t|^2+\sum |\lambda_i|^2+(|t|+\sum n_i|\lambda_i|)\Re t\sec\epsilon}{(1+\sec\epsilon)(\sum n_i|\lambda_i|+|t|)^2}.\]
\item If $\sum \lambda_i a_v(\pi_i)\in \mathbb{R}$ for all $v$
(e.g., if $\pi_i$ are self-dual for all $i$ and $\lambda_i$ are all real), $\sum \lambda_i
a_v(\pi_i)<t<0$ for $v$ in a set of places with upper Dirichlet
density at least \[\dfrac{\sum |\lambda_i|^2+t(\sum
  n_i|\lambda_i|)}{2(\sum n_i|\lambda_i|+|t|)^2}.\]
\item More generally, $\Re \sum \lambda_i a_v(\pi)<t<0$ for $v$ in a set of places with upper Dirichlet density at least \[\dfrac{\sum |\lambda_i|^2+2t(\sum
  n_i|\lambda_i|)}{4(\sum n_i|\lambda_i|+|t|)^2}.\]
\end{enumerate}
\end{theorem}

\begin{proof}
(1): Let $S$ be a finite set of places such that $\pi_{i,v}$ is
unramified for $v\notin S$ and $1\leq i\leq k$. We will apply
the approach described above to the Dirichlet series
\[Z_S(s)= \sum_{v\notin S}\frac{|\sum_{i=1}^r \lambda_i
  a_v(\pi_i)-t|^2}{q_v^s}=\sum_{v\notin S}\frac{|\sum_{i=1}^r \lambda_i
  a_v(\pi_i)|^2}{q_v^s}-2\Re t\sum_{v\notin S} \frac{\sum \lambda_i a_v(\pi_i)}{q_v^s}+|t|^2\sum_{v\notin S}\frac{1}{q_v^s}.\]
First, as $\pi_i$ is cuspidal it follows that $\log L_S(\pi_i, s)=\displaystyle \sum_{v\notin S}\sum_{k\geq 1}\frac{a_{v,k}(\pi_i)}{kq_v^{ks}}$
converges at $s=1$. As $\pi_{i,v}$ is tempered for $v\notin S$
it follows that the above sum converges for $k\geq 2$ and
therefore $\displaystyle \sum_{v\notin
  S}\frac{a_v(\pi_i)}{q_v^s}$ converges at $s=1$. We conclude
that $\displaystyle \overline{\mathcal{D}}(Z_S(s)) = \overline{\mathcal{D}}(\sum_{v\notin S}\frac{|\sum_{i=1}^r \lambda_i
  a_v(\pi_i)|^2}{q_v^s})+|t|^2$.

For $i\neq j$, $\pi_i\not\simeq \pi_j$ and therefore the function $\displaystyle \mathcal{L}_{S}(s)= \sum_{i,j=1}^r \lambda_i \overline{\lambda_j} \log L_S(\pi_i\times \pi_j^\vee,s)$
converges in the region $\Re s>1$ and has a simple pole at $s=1$
with residue $\displaystyle \sum_{i=1}^k |\lambda_i|^2$. Expanding, we deduce that
\[\mathcal{L}_S(s) =\sum_{v\notin S}\sum_{k=1}^\infty
\frac{\displaystyle |\sum_{i=1}^r \lambda_i a_{v,k}(\pi_i)|^2}{k q_v^{ks}}.\]
Since $\pi_{i,v}$ is tempered it follows that $|a_{v,k}(\pi_i)|=|\sum
\alpha_{v,j}^k(\pi_i)|\leq n_i$ for each $i$. We deduce that
\[\sum_{v\notin S}\sum_{k\geq 2}\frac{\displaystyle |\sum_{i=1}^r \lambda_ia_{v,k}(\pi_i)|^2}{k q_v^{ks}}\leq \left(\sum |\lambda_i|n_i\right)^2\sum_{v\notin
  S}\sum_{k\geq 2}\frac{1}{q_v^{ks}},\]
which converges at $s=1$. We conclude that $\displaystyle \overline{\mathcal{D}}(\sum_{v\notin S}\frac{|\sum_{i=1}^r \lambda_i
  a_v(\pi_i)|^2}{q_v^s})=\overline{\mathcal{D}}(\mathcal{L}_S(s))=\sum_{i=1}^r |\lambda_i|^2$.

Let $\mathcal{F}=\{v\notin S\mid \arg(\sum \lambda_i a_v(\pi_i)-t)\notin
(-\epsilon,\epsilon)\}$. We will use the fact that if
$\arg(z)\in (-\epsilon,\epsilon)$ then \[|z|^2\leq |z|
\sec(\epsilon) \Re(z).\] Let $\overline{\mathcal{F}}=\{v\notin S\cup \mathcal{F}\}$. We denote $B= \sum |\lambda_i|n_i+|t|$, in which case $|\sum \lambda_i a_v(\pi_i)-t|\leq B$. It follows that
\begin{align*}
\sum_{v\in \overline{\mathcal{F}}}\frac{|\sum \lambda_i
  a_v(\pi_i)-t|^2}{q_v^s}&\leq B\sec(\epsilon)\Re\sum_{v\in
  \overline{\mathcal{F}}}\frac{\sum \lambda_i
  a_v(\pi_i)-t}{q_v^s}\\
&= B\sec(\epsilon)\Re\sum_{v\notin S}\frac{\sum \lambda_i
  a_v(\pi_i)-t}{q_v^s}-B\sec(\epsilon)\Re\sum_{v\in \mathcal{F}}\frac{\sum \lambda_i
  a_v(\pi_i)-t}{q_v^s}\\
&\leq B\sec(\epsilon)\Re\sum_{v\notin S}\frac{\sum \lambda_i
  a_v(\pi_i)-t}{q_v^s}+B^2\sec(\epsilon)\sum_{v\in \mathcal{F}}\frac{1}{q_v^s}
\end{align*}
Putting everything together we get
\begin{align*}
  \sum |\lambda_i|^2+|t|^2&=\overline{\mathcal{D}}(Z_S(s))\\
&\leq \overline{\mathcal{D}}(\sum_{v\in \mathcal{F}}\frac{|\sum
  \lambda_i
  a_v(\pi_i)-t|^2}{q_v^s})+\overline{\mathcal{D}}(\sum_{v\in
  \overline{\mathcal{F}}}\frac{|\sum \lambda_i
  a_v(\pi_i)-t|^2}{q_v^s})\\
&\leq B^2(1+\sec\epsilon) \overline{\delta}(\mathcal{F})-B\Re t\sec\epsilon
\end{align*}
and therefore:
\begin{equation*}\overline{\delta}(\mathcal{F})\geq \frac{|t|^2+B\Re t\sec\epsilon+\sum |\lambda_i|^2}{B^2(1+\sec\epsilon)}.
\end{equation*}
  
(2): When $\sum \lambda_i a_v(\pi_i)$ and $t<0$ are reals, we obtain the
desired inequality letting $\epsilon\to 0$, as real numbers have
argument $0$ or $\pi$.

(3): Apply the second part to the representations
$\pi_i$ with coefficients $\lambda_i$ and $\pi_i^\vee$ with
coefficients $\overline{\lambda_i}$. Then
\[\sum (\lambda_i a_v(\pi_i)+\overline{\lambda_i}a_v(\pi_i^\vee))=\sum (\lambda_i
a_v(\pi_i)+\overline{\lambda_i}a_v(\overline{\pi}_i))=2\sum \Re\lambda_i 
a_v(\pi_i)\]
are all real numbers. Therefore $\Re\sum \lambda_i
a_v(\pi_i)<t<0$ for $v$ in a set of places of upper Dirichlet
density at least $\dfrac{\sum |\lambda_i|^2+2t(\sum
  n_i|\lambda_i|)}{4(\sum n_i|\lambda_i|)^2}$. Indeed, the lower bound $L$ from the second part of this theorem applied to the representations $\{\pi_i, \pi_i^\vee\}$ depends on isomorphism classes among the representations $\{\pi_i, \pi_i^\vee\mid 1\leq i\leq r\}$ but, since $t<0$ and $|\sum x_i|\leq \sum |x_i|$ we get $L\geq \dfrac{\sum |\lambda_i|^2+2t(\sum
  n_i|\lambda_i|)}{4(\sum n_i|\lambda_i|)^2}$ as desired.
\end{proof}

We use Theorem \ref{t:linear-comparison} to study the
distribution of Hecke coefficients $a_v(\pi)$ where $\pi$ is a
unitary cuspidal representation 
satisfying the Ramanujan conjecture (e.g., classical and Hilbert
modular forms of regular weight) and $v$ varies in certain
Galois conjugacy classes. As Theorem \ref{t:linear-comparison} is stated for mutually non-isomorphic representations, in our first application we will impose an additional technical assumption on the representation $\pi$, which is automatically satisfied when $\pi$ is not automorphically induced.

\begin{corollary}\label{c:hecke-congruences}
Let $F$ be a number field and $\pi$ be a self-dual unitary
cuspidal representation of $\GL(n, \mathbb{A}_F)$ satisfying the
Ramanujan conjecture. Let $\mathfrak{a}$ and $\mathfrak{m}$ be
coprime ideals of $\mathcal{O}_F$, and denote by
$h_{\mathfrak{m}}$ the (narrow) ray class number of conductor
$\mathfrak{m}$. Suppose that for every character $\chi$ of the ray class group of conductor $\mathfrak{m}$ such that $\pi\simeq \pi\chi$ we have $\chi(\mathfrak{a})=1$.

Then $a_v(\pi)<t\leq 0$ for $v\equiv \mathfrak{a}\pmod{\mathfrak{m}}$ in a set of places of $F$ of upper Dirichlet density at least $\dfrac{1}{2(n+|t|)^2h_{\mathfrak{m}}}+\dfrac{tn}{2(n+|t|)^2}$.
\end{corollary}
\begin{proof}
Let $\Cl(\mathfrak{m})$ be the narrow ray class group and
$H_{\mathfrak{m}}$ the associated class field. Consider every
character $\chi$ of
$\Gal(H_{\mathfrak{m}}/F)\cong\Cl(\mathfrak{m})$ as a Hecke
character of $F$. We will apply Theorem
\ref{t:linear-comparison} to the linear combination
\[f_v=\sum_{\chi:\Cl(\mathfrak{m})\to
  \mathbb{C}^\times}\frac{\chi(\mathfrak{a}^{-1})}{h_{\mathfrak{m}}}
a_v(\chi\pi) = \begin{cases}a_v(\pi)&\textrm{if }v\equiv
\mathfrak{a}\pmod{\mathfrak{m}}\\
0&\textrm{otherwise}\end{cases}.\]
The linear combination is a real number for all $v$ as $\pi$ is
self-dual. Let $H\subset \Cl(\mathfrak{m})$ be the group
$H=\{\chi\in \widehat{\Cl}(\mathfrak{m})\mid \pi\cong\pi\chi\}$
and denote by $h$ its order. We will apply Theorem
\ref{t:linear-comparison} to the representations $\chi\pi$ for
representatives $\chi\in \widehat{\Cl}(\mathfrak{m})/H$ with
coefficients $\dfrac{h
  \chi(\mathfrak{a}^{-1})}{h_{\mathfrak{m}}}$ noting that
$\chi(\mathfrak{a}^{-1})$ is independent of the choice of
representative by assumption. We conclude that $f_v<t$ for $v$
in a set of places with upper Dirichlet density at least 
\[\dfrac{\sum\limits_{\chi\in \widehat{\Cl}(\mathfrak{m})/H}
  |h\chi(\mathfrak{a}^{-1})/h_{\mathfrak{m}}|^2+tn}{2n^2(\sum\limits_{\chi\in
    \widehat{\Cl}(\mathfrak{m})/H}
  |h\chi(\mathfrak{a}^{-1})/h_{\mathfrak{m}}|+|t|)^2}=\dfrac{h}{2(n+|t|)^2h_{\mathfrak{m}}}+\dfrac{tn}{2(n+|t|)^2}\geq \dfrac{1}{2(n+|t|)^2h_{\mathfrak{m}}}+\dfrac{tn}{2(n+|t|)^2}.\]
The
desired result follows from the fact that $f_v=0$ unless
$v\equiv \mathfrak{a}\pmod{\mathfrak{m}}$,
in which case $f_v=a_v(\pi)$.
\end{proof}

\begin{remark}
Considering upper Dirichlet density as a probability function $\Pr$, 
Corollary \ref{c:hecke-congruences} can be rephrased in terms of conditional probability, using the Chebotarev density theorem, as $$\Pr(a_v(\pi)<0\mid v\equiv \mathfrak{a} \pmod{\mathfrak{m}})\geq \dfrac{1}{2n^2}.$$
\end{remark}

In the case of $\GL(2)$ and $\GL(3)$ we are able to get additional distributional properties using the functoriality of tensor product on $\GL(2)\times\GL(2)$ and $\GL(2)\times\GL(3)$.

\begin{corollary}\label{c:hecke-congruences-quadratic}
Let $F$ be a number field and $E/F$ a finite Galois extension
such that there exists a quadratic subextension $E/L/F$ with
$\Gal(E/L)$ abelian of order $n$.

\begin{enumerate}
\item Let $\pi$ be a self-dual
unitary cuspidal representation of $\GL(d,\mathbb{A}_F)$ where
$d=2$ or $3$, such that $\pi$ satisfies the Ramanujan
conjecture. When $d=2$ we furthermore assume that $\pi$ is not dihedral. For each $t<0$ we denote by $\mathcal{F}_E(t)$ the
set of places $v$ of $F$ which {\bf split completely in $E$}
such that $a_v(\pi)<t$. Then

$$\overline{\delta}(\mathcal{F}_E(t))\geq
\frac{n+1+2td^2n^2}{2(2+|t|)^2d^2n^2}.$$

\item If $\pi$ is a non-dihedral self-dual unitary cuspidal representation of $\GL(2,\mathbb{A}_F)$ satisfying the Ramanujan conjecture then $|a_v(\pi)|>1$ for $v$ in a set of completely split places in $E$ of upper Dirichlet density at least $\dfrac{n+1}{72n^2}$.
\end{enumerate}
\end{corollary}

\begin{proof}
For each character
$\chi$ of $\Gal(E/L)$ consider the automorphic representation
$\sigma_\chi$ of $\GL(2,\mathbb{A}_F)$ given by
$\sigma_\chi=\Ind_{L/F}\chi$. When $\chi\neq 1$ the
representation $\sigma_\chi$ is cuspidal and we denote
$\Pi_\chi=\pi\otimes \sigma_\chi$ the automorphic representation of
$\GL(2d,\mathbb{A}_F)$.

As in the proof of Corollary \ref{c:hecke-congruences-quadratic} we will treat each $\nu\in \Gal(L/F)$ as a Hecke character on $F$. Note that 
\begin{align*}
\sum_{\nu\in \Gal(L/F)}a_v(\nu\pi)+\sum_{\chi\neq 1}a_v(\pi\otimes
\sigma_{\chi})=a_v(\pi)\sum_\chi
a_v(\sigma_\chi)=\begin{cases}dna_v(\pi)&v\textrm{ splits
  completely in }E\\0&\textrm{otherwise}\end{cases}
\end{align*}
is a real number for all $v$ as $\pi$ is self-dual.

We will now proceed with the case $d=2$. Then $\pi$ is not twist equivalent to any of the $\sigma_\chi$, as $pi$ is not dihedral, and therefore $\Pi_\chi$ are cuspidal representations. Moreover, since $\pi$ satisfies the Ramanujan
conjecture so do the representations $\Pi_\chi$.
We will apply Theorem \ref{t:linear-comparison}
to $\pi$, $\theta\pi$ (where $\theta$ is the quadratic character of $L/F$), and to the isomorphism classes of representations
$\Pi_\chi$, as $\chi$ varies among the nontrivial characters of
$\Gal(E/L)$. Let $\Pi_{\chi_1}, \ldots,
\Pi_{\chi_r}$ be the representatives of the isomorphism classes
of $\Pi_\chi$, and let $a_i$ be the number of $\chi$ such
that $\Pi_{\chi_i}\simeq \Pi_\chi$. In Theorem \ref{t:linear-comparison} we have 2 of the $\lambda_i=\frac{1}{dn}$ with $n_i=d=2$, the other $\lambda_i=\frac{a_i}{dn}$ with $n_i=2d=4$, giving $\sum |\lambda_i|^2=\dfrac{2+\sum a_i^2}{d^2n^2}$ and $\sum n_i|\lambda_i|=2$. We
obtain that
\begin{equation}\label{eq:lhs1}
\frac{1}{dn}\sum_{\nu\in
  \Gal(L/F)}a_v(\nu\pi)+\sum\frac{a_i}{dn}a_v(\Pi_{\chi_i})<t
\end{equation}
for $v$ in a set with upper Dirichlet density at least
\[\frac{\frac{2+\sum
  a_i^2}{d^2n^2}+2t}{2(2+|t|)^2}\geq\frac{n+1+2td^2n^2}{2(2+|t|)^2d^2n^2},\]
the inequality following from Cauchy-Schwarz as $\sum a_i=n-1$.
The desired result then follows from the fact that the LHS of \eqref{eq:lhs1} can only be $<0$ when $v$ splits completely in $E$, in which case it is $a_v(\pi)$.

In the case $d=3$ the analysis is complicated by the fact that
$\Ind_{L/F}\chi\otimes\pi$ need not be cuspidal when $\pi$ is a
cuspidal representation on $\GL(3,\mathbb{A}_F)$. However,
\cite[Theorem 3.1]{ramakrishnan-wang} implies that if $\chi\neq 1$
then $\Pi_\chi$ is either cuspidal on $\GL(6,\mathbb{A}_F)$ or
is an isobaric sum of cuspidal representations on
$\GL(3,\mathbb{A}_F)$. In the former case $\Pi_\chi$ satisfies
the Ramanujan conjecture, as $\pi$ does, and in the latter case
the two cuspidal representations satisfy the Ramanujan
conjecture.

As in the previous case the lowest density bound occurs when the
representations $\Pi_\chi$ are all non-isomorphic, after an
application of Cauchy-Schwarz. We group the cuspidal $\Pi_\chi$ into groups of isomorphic representations on $\GL(2d)$ of size $a_1,\ldots, a_k$. If $\Pi_\chi$ is not cuspidal we write $\Pi_\chi=\Pi^1_\chi\boxplus\Pi^2_\chi$ and we group the $\Pi_\chi^j$ into groups of isomorphic representations on $\GL(d)$ of size $b_1,\ldots,b_\ell$, noting that $\sum a_i+\frac{1}{2}\sum b_i=n-1$. In Theorem
\ref{t:linear-comparison} we take 2 of the $\lambda_i=\frac{1}{dn}$ with $n_i=d=3$, the other $\lambda_i$'s being $\lambda_i=\frac{a_i}{dn}$ with $n_i=2d$ and $\lambda_i=\frac{b_i}{dn}$ with $n_i=d$, giving $\sum |\lambda_i|^2=\frac{1}{d^2n^2}(2+\sum a_i^2+\sum b_i^2)$ and $\sum n_i|\lambda_i|=2$. We again obtain that
\begin{equation}\label{eq:lhs2}
\frac{1}{dn}\sum_{\nu\in\Gal(L/F)}a_v(\nu
\pi)+\sum_{\Pi_\chi\textrm{
    cuspidal}}\frac{1}{dn}a_v(\Pi_\chi)+\sum_{\Pi_\chi\cong
  \Pi^1_\chi\boxplus\Pi^2_\chi}\frac{1}{dn}(a_v(\Pi^1_\chi)+a_v(\Pi^2_\chi))<t
\end{equation}
for $v$ in a set with upper Dirichlet density at least
\[\frac{\frac{2+\sum a_i^2+\sum
    b_i^2}{d^2n^2}+2t}{2(2+|t|)^2}\geq\frac{n+1+2td^2n^2}{2(2+|t|)^2d^2n^2},\]
as Cauchy-Schwarz implies that $\sum a_i^2+\sum b_i^2\geq n-1$. 
The result then follows from the fact that the LHS of \eqref{eq:lhs2} can only be $<0$ when $v$ splits completely in $E$ in which case the LHS is $a_v(\pi)$.

The last statement follows from applying the above bounds to the cuspidal representation $\Sym^2\pi$ of $\GL(3,\mathbb{A}_F)$ when $t=0$.
\end{proof}

\begin{corollary}\label{c:hecke-congruences-cubic}
Let $F$ be a number field and $E/F$ a finite Galois extension such that there exists a cubic Galois subextension $E/L/F$ with $\Gal(E/L)$ abelian of order $n$. Let $\pi$ be a non-dihedral self-dual unitary cuspidal representation of $\GL(2,\mathbb{A}_F)$ which satisfies the Ramanujan conjecture. For each $t<0$ we denote by $\mathcal{F}_E(t)$ the set of places $v$ of $F$ which {\bf split completely in $E$} such that $a_v(\pi)<t$. Then $$\overline{\delta}(\mathcal{F}_E(t))\geq \dfrac{n+2+18tn^2}{18(2+|t|)^2n^2}.$$
\end{corollary}

\begin{proof}
For each nontrivial character $\chi$ of $\Gal(E/L)$ consider the
cuspidal representation $\sigma_\chi=\Ind_{L/F}\chi$, a cuspidal
representation of $\GL(3,\mathbb{A}_F)$, and consider the
cuspidal representation $\Pi_\chi=\sigma_\chi\otimes \pi$ of
$\GL(6,\mathbb{A}_F)$, which necessarily satisfies the Ramanujan
conjecture. Note that
\begin{align*}
\sum_{\nu\in \Gal(L/F)}a_v(\nu\pi)+\sum_{\chi\neq 1}a_v(\pi\otimes
\sigma_{\chi})=a_v(\pi)\sum_\chi
a_v(\sigma_\chi)=\begin{cases}3na_v(\pi)&v\textrm{ splits
  completely in }E\\0&\textrm{otherwise}\end{cases}
\end{align*}
is a real number for all $v$ as $\pi$ is self-dual. Theorem
\ref{t:linear-comparison} then implies that
\[\frac{1}{3n}\sum_{\nu\in \Gal(L/F)}a_v(\nu\pi)+\frac{1}{3n}\sum_{\chi\neq 1}a_v(\pi\otimes
\sigma_{\chi})<t\]
for $v$ in a set of place with upper Dirichlet density at least $\dfrac{\frac{n+2}{9n^2}+2t}{2(2+|t|)^2}=\dfrac{n+2+18tn^2}{18(2+|t|)^2n^2}$.

\end{proof}

\begin{example}
Suppose $f$ is a non-CM holomorphic newform with real Fourier coefficients. Since the Ramanujan conjecture is known in this case, we obtain the following distribution results:
\begin{enumerate}
\item Corollary \ref{c:hecke-congruences} implies that $a_p(f)<0$  for $p\equiv a\pmod{n}$ in a set of upper Dirichlet density $\geq \dfrac{1}{8\varphi(n)}$.
\item Recall that a prime $p$ is of the form $m^2+27n^2$ if and only if $p$ splits completely in $E=\mathbb{Q}(\zeta_3,\sqrt[3]{2})$. Then Corollary \ref{c:hecke-congruences-quadratic} implies that $a_p(f)<0$ for $p$ of the form $m^2+27n^2$ varying in a
set of upper Dirichlet density at least $\dfrac{1}{72}$. 
\end{enumerate}

\end{example}

\begin{example}
Suppose $X$ is a smooth projective genus 2 curve over
$\mathbb{Q}$ such that $\End(\Jac(X))=\mathbb{Z}$, and $a_p = p+1-|X(\mathbb{F}_p)|$. A recent result
of Boxer, Calegari, Gee, and Pilloni \cite{boxer-calegari-gee-pilloni} implies that there exists a
totally real field $F$ and cuspidal representation $\pi$ of
$\GSp(4,\mathbb{A}_F)$ such that
$L(\rho_{X,q}|_{G_F},s)=L(\pi,\spin,s)$, where for a fixed prime
$q$, $\rho_{X,q}=T_q \Jac(X)\otimes \mathbb{Q}_q$ is the Tate
module of the Jacobian of $X$. The representation $\pi$ then satisfies the Ramanujan conjecture as $\rho_{X,q}$ is realized in the etale cohomology of $\Jac(X)$.

The density of primes $p\equiv a\pmod{n}$ such that $a_p<0$ is then equal to the density of totally split places $v$ of $F$ such that $N(v)\equiv a\pmod{n}$ and $a_v(\pi)<0$, where $a_v(\pi)$ is the Hecke coefficient of the functorial transfer of $\pi$ to an automorphic representation of $\GL(4,\mathbb{A}_F)$ which, by Corollary \ref{c:hecke-congruences}, is at least $\dfrac{1}{32\varphi(n)}$.
\end{example}

We now turn to products of Hecke coefficients. The following result generalizes \cite[Theorem 1.2]{chiriac:hecke} with $\lambda_1=1$, $\lambda_2=-1$, all other coefficients being 0. 

\begin{corollary}\label{t:quadratic-comparison}
Let $F$ be a number field. Suppose $\pi_1,\ldots, \pi_a,
\sigma_1,\sigma_2$ are pairwise twist inequivalent self-dual non-dihedral unitary cuspidal
representations of $\GL(2,\mathbb{A}_F)$. Let $\tau_1$ be a self-dual unitary cuspidal representation of $\GL(2,
\mathbb{A}_F)$ which is not a twist of $\sigma_i$ and $\tau_2$ a self-dual unitary cuspidal representation of $\GL(3,
\mathbb{A}_F)$ which is not a twist of $\Ad(\sigma)$. Consider
coefficients $\lambda_1,\ldots, \lambda_a, \nu_1,\nu_2$, and $t\in \mathbb{R}$. Suppose the representations $\pi_i, \sigma_j, \tau_j$ satisfy the Ramanujan conjecture. Then the set \[\{v\mid \sum_{i=1}^a \lambda_i a_v(\pi_i)^2 +\sum_{j=1}^2\nu_j
a_v(\sigma_j)a_v(\tau_j)<t\}\] has upper
Dirichlet density at least
\[\dfrac{(t-A)^2+B+(t-A)(|t-A|+3C+4|\nu_1|+6|\nu_2|)}{2(3C+4|\nu_1|+6|\nu_2|+|t|)^2},\]
where $A=\sum\limits_{i=1}^a \lambda_i$, $B=\sum\limits_{i=1}^a
\lambda_i^2+ \sum\limits_{j=1}^2\nu_j^2$, and
$C=\sum\limits_{i=1}^a |\lambda_i|$.

In particular, if $\pi$ and $\sigma$ are twist inequivalent unitary cuspidal representations of $\GL(2,\mathbb{A}_F)$ which satisfy the Ramanujan conjecture then $\sign a_v(\pi)\neq \sign a_v(\sigma)$ for $v$ in a set of places of upper Dirichlet density at least $1/32$.
\end{corollary}

\begin{proof}
We denote by $\Pi_i = \Sym^2\pi_i$, which is a cuspidal representation of
$\GL(3,\mathbb{A}_F)$, since $\pi_i$ is assumed not to be dihedral. Similarly,
$\Sigma_i = \sigma_i\otimes\tau_i$ is an automorphic
representation on $\GL(4,\mathbb{A}_F)$, resp. $\GL(6,
\mathbb{A}_F)$. The hypotheses imply that $\Sigma_i$ is
cuspidal and that $\Pi_i$ and $\Sigma_j$ satisfy the Ramanujan
conjecture. Then $$\sum_{i=1}^a \lambda_i a_v(\pi_i)^2 +\sum_{j=1}^2 \nu_j
a_v(\sigma_j)a_v(\tau_j)= \sum_{i=1}^a \lambda_i a_v(\Pi_i) + \sum_{j=1}^2
\nu_j(\Sigma_j) +\sum_{i=1}^a \lambda_i.$$ The desired result would
follow from Theorem \ref{t:linear-comparison} using $t-\sum_{i=1}^a
\lambda_i$ if we showed that the representations $\Pi_i$ and
$\Sigma_j$ are pairwise non-isomorphic. First, $\Pi_i$, $\Sigma_1$, and
$\Sigma_2$ are representations on groups of different ranks so
they cannot be isomorphic. Since $\pi_i$ are twist inequivalent
it follows that $\Pi_i$ are mutually non-isomorphic
\cite[Appendix]{duke-kowalski}.
\end{proof}

In the context of unitary cuspidal representations of
$\GL(2,\mathbb{A}_F)$ with trivial central characters we are able
to answer a finer question, concerning the distribution of linear combinations of the Hecke coefficients in a given interval. For classical modular forms, a similar question was studied by Matom\"aki in \cite{matomaki:signs}.

\begin{proposition} \label{t:linear-comparison-two}
Let $F$ be a number field. Let $\pi_1$ and $\pi_2$ be two twist-inequivalent unitary cuspidal representations of $\GL(2, \mathbb{A}_F)$ with trivial central characters, satisfying the Ramanujan conjecture. Suppose that $\pi_1$ and $\pi_2$ are not dihedral or tetrahedral. Let $\lambda_1, \lambda_2, a, b\in \mathbb{R}$ with $a<b$. Then the set $$\mathcal{F}=\{v\mid a<\lambda_1 a_v(\pi_1)+ \lambda_2 a_v(\pi_2) < b\}$$ has upper Dirichlet density $\overline{\delta}(\mathcal{F})$ at least 
\begin{align}
 &\frac{2(\lambda_1^4+\lambda_2^4)+6\lambda_1^2\lambda_2^2+(\lambda_1^2+\lambda_2^2)((a+b)^2+2ab)+a^2b^2}{2\big(4(|\lambda_1|+|\lambda_2|)^2+2(|a|+|b|)(|\lambda_1|+|\lambda_2|)+|ab|)\big)^2} \notag \\
 &-\frac{\lambda_1^2+\lambda_2^2+ab}{2\big(4(|\lambda_1|+|\lambda_2|)^2+2(|a|+|b|)(|\lambda_1|+|\lambda_2|)+|ab|)\big)}. \notag 
\end{align}
\end{proposition}

\begin{proof}
Consider the Dirichlet series
\[Z_S(s) = \sum_{v\notin S}\frac{((\sum_{i=1}^2 \lambda_i
  a_v(\pi_i)-a)(\sum_{i=1}^2 \lambda_i a_v(\pi_i)-b))^2}{q_v^s},\]
where $S$ is a finite set of places outside of which $\pi_1$ and $\pi_2$ are unramified. Then 

\begin{align}
Z_S(s)&=\sum_{v\notin S} \frac{(\lambda_1 a_v(\pi_1)+ \lambda_2 a_v(\pi_2))^4}{q_v^s}+\left((a+b)^2+2ab\right)\sum_{v\notin S} \frac{(\lambda_1 a_v(\pi_1)+ \lambda_2 a_v(\pi_2))^2}{q_v^s}+a^2b^2\sum_{v\notin S} \frac{1}{q_v^s} \notag \\
&-2ab(a+b)\sum_{v\notin S} \frac{\lambda_1 a_v(\pi_1)+ \lambda_2 a_v(\pi_2)}{q_v^s}-2(a+b)\sum_{v\notin S} \frac{(\lambda_1 a_v(\pi_1)+ \lambda_2 a_v(\pi_2))^3}{q_v^s}. \notag 
\end{align}
Since $\pi_i$ ($1\leq i \leq 2$) is not dihedral or tetrahedral, the representations $\pi_i$, $\Sym^2\pi_i$, and $\Sym^3 \pi_i$ are all cuspidal. Moreover, since $\pi_1$ and $\pi_2$ are twist-inequivalent, $\pi_1\otimes\pi_2$ is also cuspidal. It follows that
their $L$-functions are entire, so the last two terms in the
above identity are bounded. Moreover, by the properties of the
Rankin-Selberg convolution: \[ \sum_{v\notin S} \frac{(\lambda_1
  a_v(\pi_1)+ \lambda_2
  a_v(\pi_2))^2}{q_v^s}=(\lambda_1^2+\lambda_2^2)\sum_{v\notin
  S} \frac{1}{q_v^s}+O(1).\] 
By the cuspidality of $\pi_1\otimes\pi_2$ and the properties of Rankin-Selberg $L$-functions:
\[\sum_{v\notin S} \frac{(\lambda_1 a_v(\pi_1)+ \lambda_2 a_v(\pi_2))^4}{q_v^s}=(2\lambda_1^4+2\lambda_1^4+6\lambda_1^2\lambda_2^2)\sum_{v\notin S} \frac{1}{q_v^s}+O(1).\] Combining the previous two estimates we get

\[Z_S(s)=(2\lambda_1^4+2\lambda_1^4+6\lambda_1^2\lambda_2^2+(\lambda_1^2+\lambda_2^2)((a+b)^2+2ab)+a^2b^2)\sum_{v\notin S} \frac{1}{q_v^s}+O(1).\] 

\noindent At the same time, the Ramanujan bound implies that 
\[\left|(\sum_{i=1}^2 \lambda_i
  a_v(\pi_i)-a)(\sum_{i=1}^2 \lambda_i a_v(\pi_i)-b) \right|\leq 4(|\lambda_1|+|\lambda_2|)^2+2\left(|a|+|b|\right)\left(|\lambda_1|+|\lambda_2|\right)+|ab|.\] Therefore, using the approach outlined in the beginning of Section~\ref{s:no-rc-gl2} with 
\begin{align}
f_v&=(\sum_{i=1}^2 \lambda_i a_v(\pi_i)-a)(\sum_{i=1}^2 \lambda_i a_v(\pi_i)-b), \notag \\
B&=4(|\lambda_1|+|\lambda_2|)^2+2(|a|+|b|)(|\lambda_1|+|\lambda_2|)+|ab|, \notag \\
m&=2\lambda_1^4+2\lambda_2^4+6\lambda_1^2\lambda_2^2+(\lambda_1^2+\lambda_2^2)((a+b)^2+2ab)+a^2b^2, \notag \\
M&=\lambda_1^2+\lambda_2^2+ab, \notag 
\end{align}  the conclusion follows. 
\end{proof}

\begin{remark}\label{abs_val} Taking $\lambda_1=1$, $\lambda_2=-1$ and $b=-a>0$ we obtain that $|a_v(\pi_1)-a_v(\pi_2)|\leq b$ for $v$ in a set of places $\mathcal{F}$ with:
$$\overline{\delta}(\mathcal{F})\geq \frac{b^4+4b^3+5b^2-8b-11}{(b^2+8b+16)^2},$$ which is positive when $b>1.3371$.
\end{remark}

\bibliographystyle{siam}
\bibliography{biblio}

\begin{thebibliography}{10}

\bibitem{10-author}
{\sc P.~B. Allen, F.~Calegari, A.~Caraiani, T.~Gee, D.~Helm, B.~V.~L. Hung,
  J.~Newton, P.~Scholze, R.~Taylor, and J.~A. Thorne}, {\em Potential
  automorphy over cm fields}, 2018.

\bibitem{banaszak-kedlaya}
{\sc G.~Banaszak and K.~S. Kedlaya}, {\em An algebraic {S}ato-{T}ate group and
  {S}ato-{T}ate conjecture}, Indiana Univ. Math. J., 64 (2015), pp.~245--274.

\bibitem{sato-tate:hmf}
{\sc T.~Barnet-Lamb, T.~Gee, and D.~Geraghty}, {\em The {S}ato-{T}ate
  conjecture for {H}ilbert modular forms}, J. Amer. Math. Soc., 24 (2011),
  pp.~411--469.

\bibitem{blomer-brumley}
{\sc V.~Blomer and F.~Brumley}, {\em On the {R}amanujan conjecture over number
  fields}, Ann. of Math. (2), 174 (2011), pp.~581--605.

\bibitem{boxer-calegari-gee-pilloni}
{\sc G.~Boxer, F.~Calegari, T.~Gee, and V.~Pilloni}, {\em Abelian surfaces over
  totally real fields are potentially modular}, 2018.

\bibitem{chiriac:hecke}
{\sc L.~Chiriac}, {\em Comparing {H}ecke eigenvalues of newforms}, Arch. Math.
  (Basel), 109 (2017), pp.~223--229.

\bibitem{sato-tate-2}
{\sc L.~Clozel, M.~Harris, and R.~Taylor}, {\em Automorphy for some {$l$}-adic
  lifts of automorphic mod {$l$} {G}alois representations}, Publ. Math. Inst.
  Hautes \'Etudes Sci.,  (2008), pp.~1--181.
\newblock With Appendix A, summarizing unpublished work of Russ Mann, and
  Appendix B by Marie-France Vign\'eras.

\bibitem{duke-kowalski}
{\sc W.~Duke and E.~Kowalski}, {\em A problem of {L}innik for elliptic curves
  and mean-value estimates for automorphic representations}, Invent. Math., 139
  (2000), pp.~1--39.
\newblock With an appendix by Dinakar Ramakrishnan.

\bibitem{gelbart-jacquet}
{\sc S.~Gelbart and H.~Jacquet}, {\em A relation between automorphic
  representations of {${\rm GL}(2)$}\ and {${\rm GL}(3)$}}, Ann. Sci. \'Ecole
  Norm. Sup. (4), 11 (1978), pp.~471--542.

\bibitem{harris:sato-tate}
{\sc M.~Harris}, {\em Galois representations, automorphic forms, and the
  {S}ato-{T}ate conjecture}, Indian J. Pure Appl. Math., 45 (2014),
  pp.~707--746.

\bibitem{sato-tate-1}
{\sc M.~Harris, N.~Shepherd-Barron, and R.~Taylor}, {\em A family of
  {C}alabi-{Y}au varieties and potential automorphy}, Ann. of Math. (2), 171
  (2010), pp.~779--813.

\bibitem{harris-taylor}
{\sc M.~Harris and R.~Taylor}, {\em The geometry and cohomology of some simple
  {S}himura varieties}, vol.~151 of Annals of Mathematics Studies, Princeton
  University Press, Princeton, NJ, 2001.
\newblock With an appendix by Vladimir G. Berkovich.

\bibitem{kim:sym4}
{\sc H.~H. Kim}, {\em Functoriality for the exterior square of {${\rm GL}_4$}
  and the symmetric fourth of {${\rm GL}_2$}}, J. Amer. Math. Soc., 16 (2003),
  pp.~139--183.
\newblock With appendix 1 by Dinakar Ramakrishnan and appendix 2 by Kim and
  Peter Sarnak.

\bibitem{kim-shahidi}
{\sc H.~H. Kim and F.~Shahidi}, {\em Functorial products for {${\rm
  GL}_2\times{\rm GL}_3$} and the symmetric cube for {${\rm GL}_2$}}, Ann. of
  Math. (2), 155 (2002), pp.~837--893.
\newblock With an appendix by Colin J. Bushnell and Guy Henniart.

\bibitem{sound-paper}
{\sc E.~Kowalski, Y.-K. Lau, K.~Soundararajan, and J.~Wu}, {\em On modular
  signs}, Math. Proc. Cambridge Philos. Soc., 149 (2010), pp.~389--411.

\bibitem{luo-rudnick-sarnak}
{\sc W.~Luo, Z.~Rudnick, and P.~Sarnak}, {\em On the generalized {R}amanujan
  conjecture for {${\rm GL}(n)$}}, in Automorphic forms, automorphic
  representations, and arithmetic ({F}ort {W}orth, {TX}, 1996), vol.~66 of
  Proc. Sympos. Pure Math., Amer. Math. Soc., Providence, RI, 1999,
  pp.~301--310.

\bibitem{matomaki:signs}
{\sc K.~Matom\"aki}, {\em On signs of {F}ourier coefficients of cusp forms},
  Math. Proc. Cambridge Philos. Soc., 152 (2012), pp.~207--222.

\bibitem{rajan:galois}
{\sc C.~S. Rajan}, {\em On strong multiplicity one for {$l$}-adic
  representations}, Internat. Math. Res. Notices,  (1998), pp.~161--172.

\bibitem{ramakrishnan:pure}
{\sc D.~Ramakrishnan}, {\em Pure motives and automorphic forms}, in Motives
  ({S}eattle, {WA}, 1991), vol.~55 of Proc. Sympos. Pure Math., Amer. Math.
  Soc., Providence, RI, 1994, pp.~411--446.

\bibitem{ramakrishnan:taylor}
\leavevmode\vrule height 2pt depth -1.6pt width 23pt, {\em A refinement of the
  strong multiplicity one theorem for {${\rm GL}(2)$}. {A}ppendix to:
  ``{$l$}-adic representations associated to modular forms over imaginary
  quadratic fields. {II}'' [{I}nvent.\ {M}ath.\ {\bf 116} (1994), no.\ 1-3,
  619--643; {MR}1253207 (95h:11050a)] by {R}. {T}aylor}, Invent. Math., 116
  (1994), pp.~645--649.

\bibitem{ramakrishnan:coefficients}
\leavevmode\vrule height 2pt depth -1.6pt width 23pt, {\em On the coefficients
  of cusp forms}, Math. Res. Lett., 4 (1997), pp.~295--307.

\bibitem{ramakrishnan:tensor}
\leavevmode\vrule height 2pt depth -1.6pt width 23pt, {\em Modularity of the
  {R}ankin-{S}elberg {$L$}-series, and multiplicity one for {${\rm SL}(2)$}},
  Ann. of Math. (2), 152 (2000), pp.~45--111.

\bibitem{ramakrishnan-wang}
{\sc D.~Ramakrishnan and S.~Wang}, {\em A cuspidality criterion for the
  functorial product on {$\rm GL(2)\times GL(3)$} with a cohomological
  application}, Int. Math. Res. Not.,  (2004), pp.~1355--1394.

\bibitem{sage}
{\sc {SageMath, Inc.}}, {\em {CoCalc Collaborative Computation Online}}, 2016.
\newblock {\tt https://cocalc.com/}.

\bibitem{shahidi:symmetric}
{\sc F.~Shahidi}, {\em Symmetric power {$L$}-functions for {${\rm GL}(2)$}}, in
  Elliptic curves and related topics, vol.~4 of CRM Proc. Lecture Notes, Amer.
  Math. Soc., Providence, RI, 1994, pp.~159--182.

\bibitem{sato-tate-3}
{\sc R.~Taylor}, {\em Automorphy for some {$l$}-adic lifts of automorphic mod
  {$l$} {G}alois representations. {II}}, Publ. Math. Inst. Hautes \'Etudes
  Sci.,  (2008), pp.~183--239.

\bibitem{walji:multiplicity-one}
{\sc N.~Walji}, {\em Further refinement of strong multiplicity one for {$\rm
  GL(2)$}}, Trans. Amer. Math. Soc., 366 (2014), pp.~4987--5007.

\bibitem{walji:large}
{\sc N.~Walji}, {\em On the occurrence of large positive hecke eigenvalues for
  {${\rm GL}(2)$}}, 2016.

\end{thebibliography}

\end{document}